\documentclass[reqno,colorlinks=true,citecolor=magenta,linkcolor=blue]{amsart}

\usepackage{amsthm,amsmath,amsfonts,amssymb}
\usepackage{hyperref}
\usepackage{color}
\usepackage{amsmath,amsfonts}
\usepackage{mathrsfs}

\usepackage{calligra}
\usepackage[T1]{fontenc}
\usepackage[shortlabels]{enumitem}

\usepackage{subfigure,graphicx}
\usepackage{hyperref,times} %for contents
\usepackage{mathrsfs}
\usepackage[numbers,compress]{natbib}
\usepackage{amssymb}
\usepackage[text={145mm,200mm},centering]{geometry} %130,200
\usepackage{inputenc}

\newtheorem{theorem}{Theorem}[section]
\newtheorem{lemma}[theorem]{Lemma}
\newtheorem{corollary}[theorem]{Corollary}

\theoremstyle{definition}
\newtheorem{definition}[theorem]{Definition}

\theoremstyle{remark}
\newtheorem{remark}[theorem]{Remark}
\theoremstyle{definition}\newtheorem{assumption}{Assumption}[section]
\numberwithin{equation}{section}

\def\h{\mathcal H}

\def\b{\mathcal B}

\def\cald {\mathcal D}
\def\a{\mathcal A}
\def\E{\mathcal E}
\def\d{\mathrm d}
\def\e{\epsilon}

\def\r{\mathbb R}

\begin{document}

%\vskip 0.4in

\title{\bfseries Exponential attractor for the  viscoelastic wave model  with time-dependent memory kernels
  \footnote{Supported by National
Natural Science Foundation of China (No.11671367). *Corresponding author: Zhijian Yang, e-mail: liyn@hrbeu.edu.cn (Y. Li), yzjzzut@tom.com (Z. Yang)} }

\author[~]{Yanan Li$^1$, \ \ Zhijian Yang$^{2, *}$\\
${}^1$ College of Mathematical Sciences, Harbin Engineering University, 150001, China\\
${}^2$ School of Mathematics and Statistics, Zhengzhou
University,  450001,   China}

\date{}
\maketitle \thispagestyle{empty} \setcounter{page}{1}

\begin{abstract} The paper is concerned with  the   exponential attractors for the   viscoelastic wave model  in $\Omega\subset \mathbb R^3$:
\begin{align*}
 u_{tt}-h_t(0)\Delta u-\int_0^\infty\partial_sh_t(s)\Delta u(t-s)\mathrm ds+f(u)=h,
\end{align*}
 with time-dependent memory kernel  $h_t(\cdot)$  which is used to model  aging phenomena of the material. Conti et al     \cite{Pata1,Pata2} recently provided the  correct mathematical setting for  the model  and  a well-posedness result within the novel theory of dynamical systems acting on
time-dependent spaces, recently established by Conti, Pata and Temam   \cite{Pata}, and proved the existence and the regularity of the time-dependent
global attractor.  In this work, we further   study the   existence of the time-dependent  exponential attractors as well as their  regularity. We    establish  an abstract  existence criterion via  quasi-stability method  introduced  originally  by Chueshov and Lasiecka \cite{Chueshov2004},  and on the basis of the theory and technique developed in  \cite{Pata1,Pata2}  we further provide a new method to  overcome the difficulty of the lack of further regularity  to show  the existence of the time-dependent  exponential attractor. And these techniques can be used to tackle  other hyperbolic models.
\end{abstract}

\vspace{.08in} \noindent \textbf{Keywords}: Viscoelastic wave model; time-dependent memory kernel; exponential attractors;     time-dependent phase spaces; longtime behavior of solutions.

\vspace{.08in}  \noindent \textbf{2020 Mathematics Subject Classification:} 37L30, 37L45, 35B40, 35B41, 35L10.

\section{Introduction}

In this paper, we investigate  the existence of the exponential attractors for the  following viscoelastic wave model  with  time-dependent memory kernel
\begin{align}
 u_{tt}-h_t(0)\Delta u-\int_0^\infty\partial_sh_t(s)\Delta u(t-s)\mathrm ds+f(u)=h  \ \ \hbox{in} \ \ \Omega\times (\tau,+\infty),\label{a1}
\end{align}
where $\Omega\subset \mathbb R^3$ is a bounded domian  with the smooth boundary $\partial \Omega$, the time-dependent function
\begin{align}\label{a2}
h_t(s)=k_t(s)+k_\infty,\ \ k_\infty>0,  \ \ u(\tau-s)=\phi_\tau(s),  \  s>0,
\end{align}
and where   $k_t(\cdot)$ is convex and summable for every fixed $t, \phi_\tau(s)$ is an assigned functions, together with the boundary and initial conditions
\begin{align} \label{a3}
u|_{\partial\Omega}=0,\ \ (u(\tau), u_t(\tau))=(u_\tau,v_\tau).
\end{align}
Model \eqref{a1} arising from the theory of viscoelasticity was proposed by Conti, Danese, Giorgi and Pata \cite{Pata1} to   describe the dynamics of aging materials because   the memory kernel $h_t(\cdot)$ depends on time, and  this feature allows to describe viscoelastic materials whose structural properties evolve over time, say,  materials that undergo an aging
process   which can be reasonably depicted as a loss of the elastic response (for more details, one can  see \cite{Pata1-3,Pata1-17, Pata1-28} and references therein).
 This translates into the study
of dynamical systems acting on time-dependent spaces, according to the newly established theory by
Conti, Pata and Temam\cite{Pata}, whose inspiration is on the basis of \cite{Chepyzhov, Di}.

The
presence of a time-dependent kernel introduces essential difficulties in the analysis. When the memory kernels are independent $t$, the classical method introduced by Dafermos \cite{Pata1-8,Pata1-9}  is adding a new variable $\eta$ which is generated by the right-translation semigroup acting on the history space and satisfies a specific differential equation. Under this circumstance, there exist extensive   researches on the well-posedness and the longtime   behavior of Eq. \eqref{a1}, with $h_t(s)\equiv h(s)$, see \cite{Danese, Pata1-8, Pata1-16, Pata1-19, Pata1-21, Pata1-22, Pata1-24} and references therein.
However, this approach become useless for the models with  time-dependent memory kernel because the phase spaces for the past history are time-dependent, which   causes some
problems even in the definition of the time derivative $\partial_t\eta$.  Hence it is necessary to give a new definition and  construct  some new estimates about variable $\eta$.

Conti,  Danese,  Giorgi and  Pata \cite{Pata1} introduce the time-dependent memory kernel
\begin{align*}
\mu_t(s):=-\partial_sk_t(s)=-\partial_s h_t(s),
\end{align*}
and for simplicity let $k_\infty=1, \eta_\tau(s)=u_\tau-\phi_\tau(s), s\in \r^+$. Then a simple calculation shows that problem  \eqref{a1}-\eqref{a3} reads
\begin{equation}
 \partial_{tt}u+A u+\int_0^\infty\mu_t(s)A\eta^t(s)\mathrm ds+f(u)=h,  \label{1.1}
\end{equation}
where  $A$ is the Laplacian with the Dirichlet boundary condition, with domain  $D(A)=H^2(\Omega)\cap H_0^1(\Omega)$, and
\begin{align}
\eta^t(s)=\begin{cases}u(t)-u(t-s),  s\leq t-\tau, \\
\eta_\tau(s-t+\tau)+u(t)-u_\tau, s>t-\tau,\label{1.2}
\end{cases}
\end{align}
with the initial condition
\begin{align}\label{1.3}
(u(\tau),\partial_t u(\tau), \eta^\tau)=(u_\tau,v_\tau,\eta_\tau).
\end{align}
By proposing a different notion of weak solution where the  supplementary
differential equation ruling the evolution of $\eta$ is replaced by \eqref{1.2}, and establishing  a family of integral inequalities rather than differential ones as before,
 Conti,  Danese,  Giorgi and  Pata \cite{Pata1}  first provided a global   well-posedness result for problem \eqref{1.1}-\eqref{1.3}. Then  Conti,  Giorgi and  Pata \cite{Pata2} further focused on the asymptotic behavior of the weak solutions and proved the existence and the regularity of the    time-dependent global attractor.  The authors \cite{Pata1,Pata2} developed the theory, along with the techniques  in their  works, and open the way
to the longterm analysis of the solutions for the related model with  time-dependent memory kernel.

 We mention that    when the memory kernel $h_t (s)\equiv h(s)$  (independent of $t$), Danese, Geredeli and Pata \cite{Danese} have proved the existence of exponential attractors for Eq. \eqref{a1} by using the abstract criterion given in their paper.  While the concept of exponential attractors was firstly introduced by Eden et al   \cite{Eden} in the Hilbert space (and later  in  the Banach space (cf. \cite{Efendiev1})), which have the advantage of being more stable than global attractors because they have finite fractal dimension  and attract  trajectories  at an exponential rate (cf. \cite{Efendiev2, Miranville} for a detail discussion).

However, to the best of the authors'  knowledge, there are no   results  on the  exponential attractors for viscoelastic wave  model \eqref{a1} (or \eqref{1.1}) because of  the absence of the exponential attractor theory in the time-dependent phase spaces  and the technical difficulties arising from  this kind of hyperbolic problem.

  The purpose of  this paper is to probe this question, and the motivation of this research comes from literatures \cite{Pata1,Pata2}. For convenience, we use the same    terminology used in \cite{Pata1,Pata2}. The main strategies can be summarized as follows:

  (i)  We first give  a proper notion  of the time-dependent exponential attractors for the dynamical process acting  on the time-dependent phase spaces, and  provide an abstract    existence criterion via  quasi-stability method  introduced  originally  by Chueshov and Lasiecka \cite{Chueshov2004, Chueshov2008, Chueshov2015}. This  criterion can be seen as an extension of  that in \cite{Y-Ly}, which provided an abstract result for the existence of the pullback exponential attractors.

  (ii)  We provide a new method to construct a special attracting family (rather than usual absorbing family) with higher regularity and forward invariance, and based on them to apply the abstract criterion  to problem \eqref{1.1}-\eqref{1.3} to prove the existence  and the regularity of the desired time-dependent exponential attractors.

  It is worth mentioning that  the existence of the time-dependent exponential attractor implies that the fractal dimension of the  sections of time-dependent global attractor given by \cite{Pata2} are uniformly bounded, and the application of the  abstract criterion is challenging because of
 the hyperbolicity    of   Eq. \eqref{a1} (or \eqref{1.1}), which leads to   non-further regularity of the solutions.

    The main contributions of the current paper are that we provide a new method based on   the compact attracting family to overcome the difficulty of the lack of further regularity, and to apply the  abstract criterion established in this paper   (See Theorem \ref{ea} as well as Corollary \ref{ea1} and Corollary \ref{ea2}) to
        establish the existence of the  time-dependent exponential attractor of problem \eqref{1.1}-\eqref{1.3} (see Theorem \ref{eaexist}). And this technique can be exported to tackle   other hyperbolic models.

The paper is organized as follows. In Section 2, we give the definition of the time-dependent exponential attractors and discuss their  existence criterion at an  abstract level. In Section 3,  we quote the assumptions which are same with those in \cite{Pata1,Pata2}, and state the main theorem of the paper.  In Section 4, we   first quote some known results   coming from literature \cite{Pata1,Pata2}, then based on them  we further establish some new estimates which will play key roles for our proving the main theorem.  In Section 5, we give the proof of  the main theorem.

\section{Time-dependent exponential attractors}
In this section, we first quote  some notions of the time-dependent global attractor   and the related results (cf. \cite{Pata, Pata2}), then give the definition of the time-dependent exponential attractor   as well as an abstract crirerion on its existence.

\begin{definition}  A two-parameter family of operators $\{U(t,\tau):X_\tau\rightarrow X_t\ |\ t\geq \tau, \tau\in\mathbb{R} \}$ is called  a process acting on time-dependent Banach spaces $\{X_t\}_{t\in\r}$  if (i) $U(\tau, \tau)$ is the identity map on $X_\tau$; (ii)
 $U(t,s)U(s,\tau)=U(t,\tau)$ for  all $t\geq s\geq \tau$.
\end{definition}

Let $U(t,\tau)$ be a process acting on time-dependent Banach spaces $\{X_t\}_{t\in\r}$.

\begin{definition}\label{26}  A family  $\b=\{B(t)\}_{t\in\r}$    is called a uniformly time-dependent absorbing set of the process  $U(t, \tau)$ if it is uniformly bounded, i.e.,
\begin{equation*}
 \sup_{t\in\r}\|B(t)\|_{X_t}:=\sup_{t\in\r}\sup_{\xi\in B(t)}\|\xi\|_{X_t}<+\infty,
\end{equation*}
and  for every $R>0$, there exists a    $\tau_e=\tau_e(R)\geq 0$ such that
\begin{equation*}
 U(t, \tau)\mathbb B_\tau(R)\subset B(t)\ \ \hbox{as}\ \ t-\tau \geq \tau_e,
\end{equation*}
where and in the following $\mathbb{B}_\tau(R)=\{\xi\in X_\tau| \|\xi\|_{X_\tau}\leq R\}$.
 \end{definition}

\begin{definition}\cite{Pata} A family  $\a=\{A(t)\}_{t\in\r}$    is called the time-dependent global attractor   of the process  $U(t, \tau)$  if
\begin{description}
  \item (i)   $A(t)$ is  compact   in $X_t$ for each $t\in\r$;
  \item (ii) $\a$ is pullback attracting, i.e., $\a$ is uniformly bounded and for every  uniformly bounded  family  $\cald=\{D(t)\}_{t\in\r}$,
     \begin{equation*}
     \lim_{\tau\rightarrow -\infty} \mathrm{dist}_{X_t}\left(U(t,\tau)D(\tau), A(t)\right)=0, \ \ \forall t\in\r,
         \end{equation*}
         where
      \begin{equation*}
      \mathrm{dist}_{X_t}\left(A, B\right)=\sup_{x\in A}\inf_{y\in B}\|x-y\|_{X_t}
      \end{equation*}
 is the Hausdorff semidistance of the  nonempty sets  $A, B\subset X_t$.
         \item (iii)  $\a$ is the smallest family with above mentioned properties (i) and (ii), i.e.,  if a family $\a_1=\{A_1(t)\}_{t\in\r}$ is of   properties (i) and (ii), then $A(t)\subset A_1(t)$ for all  $t\in \r$.
\end{description}
\end{definition}

Now, we define  the time-dependent exponential attractor, which is a generalization of the concept of the  pullback exponential attractor, and give its existence criterion.

\begin{definition} A uniformly bounded family $\E=\{E(t)\}_{t\in\r}$   is called a time-dependent exponential attractor    of the process  $U(t, \tau)$ if
\begin{description}
  \item (i)  $E(t)$ is  compact in $X_t$ for each $t\in\r$, and its fractal dimension in $X_t$  is uniformly bounded, i.e.,
    \begin{equation*}
      \sup_{t\in\r} \dim_f\left( E(t); X_t\right) <+\infty.
    \end{equation*}
  \item (ii)  $\E$ is semi-invariant, i.e., $U(t, \tau) E(\tau)\subset E(t)$ for all $t\geq \tau$.
  \item (iii)  There exists a positive constant $\beta$ such that for  every   uniformly bounded family $\cald=\{D(t)\}_{t\in\r}$
        \begin{equation*}
      \mathrm{dist}_{X_t}\left(U(t, t-\tau)D(t-\tau), E(t)\right)\leq C(\cald) e^{-\beta \tau}, \ \  \forall \tau\geq \tau(\cald), \ t\in \r,
         \end{equation*}
where $C(\cald), \tau(\cald)$ are positive constants depending only on $\cald$.
\end{description}
 \end{definition}
\medskip

\begin{theorem}\label{ea} Let $U(t,\tau)$ be a process acting on time-dependent Banach spaces $\{X_t\}_{t\in\r}$.
Assume that there exist a uniformly bounded family $\b=\{B(t)\}_{t\in\r}$ and a positive constant $T$ such that
\begin{description}
  \item[$(H_1)$] $B(t)$ is closed in $X_t$ for each  $t\in\r$, and
      \begin{equation}\label{3.1}
       U(t, t-\tau)B(t-\tau)\subset B(t), \ \ \forall \tau\geq T;
       \end{equation}
  \item[$(H_2)$] there exists a uniform Lipschitz constant $L_1>0$ such that
     \begin{equation}\label{3.2}
     \|U(t, t-\tau)x-U(t, t-\tau)y\|_{X_t}\leq L_1 \|x-y\|_{X_{t-\tau}}
     \end{equation}
    for all $x, y\in B(t-\tau)$, $\tau\in [0, T]$ and $t\in\r$;
  \item[$(H_3)$] there exist a Banach space $Z$ with compact seminorm $n_Z(\cdot)$, and a mapping $K_t: B(t-T) \rightarrow Z$ for each $t\in\r$ such that for any $x, y\in B(t-T)$,
   \begin{align}
   & \|K_t x-K_t y\|_Z\leq L\|x-y\|_{X_{t-T}}, \label{3.3}\\
   & \|U(t, t-T)x-U(t, t-T)y\|_{X_t}\leq \eta \|x-y\|_{X_{t-T}}+n_Z\left(K_t x-K_t y\right), \label{3.4}
   \end{align}
  where $\eta\in (0, 1/2)$, $L>0$ are constants independent of $t$.
\end{description}
Then, there exists a semi-invariant family $\E=\{E(t)\}_{t\in\r}$ possessing the following properties:
\begin{description}
  \item (i)  the section $E(t)\subset B(t)$ is compact in $X_t$ for each $t\in\r$  and its fractal dimension in $X_t$ is uniformly bounded, i.e.,
       \begin{equation}\label{3.5}
  \sup_{t\in \mathbb{R}} \mathrm{dim}_f\left(E(t); X_t\right)\leq \Big[\ln\big(\frac{1}{2\eta}\big)\Big]^{-1}\ln m_Z\Big(\frac{2L}{\eta}\Big)<+\infty,
\end{equation}
where $m_Z(R)$ is the maximal number of elements $z_i$ in the ball $\{z\in Z|\|z\|_Z\leq R\}$ such that $n_Z(z_i-z_j)>1$, $i \neq j$;
  \item (ii)  there exist  positive constants $\beta$, $C$ and $\tau_0$ such that
      \begin{equation}\label{3.6}
      \mathrm{dist}_{X_t}\left(U(t, t-\tau)B(t-\tau), E(t)\right)\leq Ce^{-\beta \tau}\ \ \hbox{as}\ \ \tau\geq \tau_0, \forall t\in\r.
       \end{equation}
 \end{description}
\end{theorem}

\begin{corollary}\label{ea1} Let the assumptions of Theorem \ref{ea} be valid. If the family $\b=\{B(t)\}_{t\in\mathbb{R}}$ is a uniformly time-dependent absorbing set of the process  $U(t,\tau)$,  then the family $\E=\{E(t)\}_{t\in\mathbb{R}}$ given in Theorem \ref{ea} is a time-dependent exponential attractor of the process  $U(t,\tau)$.
 \end{corollary}

\begin{proof}
For any uniformly bounded family $\cald=\{D(t)\}_{t\in\mathbb{R}}$, there exist  positive constants $R$ and $e(R)$ such that $D(t)\subset \mathbb{B}_t(R)$ for all $t\in \mathbb{R}$ and
\begin{equation*}
  U(t, t-\tau)D(t-\tau)\subset U(t, t-\tau)\mathbb{B}_{t-\tau}(R)\subset B(t), \ \ \forall \tau\geq e(R).
\end{equation*}
Thus, it follows from \eqref{3.6} that, for every $t\in\mathbb{R}$, $\tau\geq e(R)+\tau_0$,
\begin{equation*}
  \begin{split}
    &\mathrm{dist}_{X_t}\left(U(t,t-\tau)D(t-\tau), E(t)\right)\\
  \leq\ & \mathrm{dist}_{X_t}\left(U(t, t-\tau+e(R))U(t-\tau+e(R), t-\tau)D(t-\tau), E(t)\right) \\
       \leq\ & \mathrm{dist}_{X_t}\left(U(t, t-\tau+e(R))B(t-\tau+e(R)), E(t)\right)\\
     \leq\ & Ce^{\beta e(R)}e^{-\beta\tau}.
  \end{split}
\end{equation*}
This completes the proof.
\end{proof}

\begin{corollary}\label{ea2} Let the assumptions of Theorem \ref{ea} be valid. If the process $U(t,\tau)$ possesses    a uniformly  time-dependent absorbing set $\{\mathbb{B}_t(R_1)\}_{t\in\mathbb{R}}$  possessing the following properties:
\begin{enumerate}[$(i)$]
 \item there is a positive constant $R_2>R_1$ such that $B(t)\subset \mathbb{B}_t(R_2)$ for all $t\in\r$ and formula \eqref{3.2} holds on $\mathbb{B}_t(R_2)$, that is
\begin{equation}\label{3.8}
  \|U(t, t-\tau)x- U(t, t-\tau)y\|_{X_t}\leq L_1\|x-y\|_{X_{t-\tau}}
  \end{equation}
  for all $x,y\in \mathbb{B}_{t-\tau}(R_2)$, $\tau\in [0, T]$ and $t\in\mathbb{R}$;
  \item there exist positive constants $\kappa$, $\tau_1$ and $C_0$ such that
\begin{equation}\label{3.7}
  \mathrm{dist}_{X_t}\left(U(t, t-\tau)\mathbb{B}_{t-\tau}(R_1), B(t)\right)\leq C_0 e^{-\kappa \tau}, \ \ \forall t\in \mathbb{R}, \tau\geq \tau_1.
\end{equation}
\end{enumerate}
 Then the family $\E=\{E(t)\}_{t\in\mathbb{R}}$ given in Theorem \ref{ea} is a time-dependent exponential attractor of the process  $U(t,\tau)$.
 \end{corollary}
\begin{proof}By the definition of uniformly  time-dependent absorbing set, there exists  a positive constant $e(R_1)>\tau_1$ such that for any given   $\theta\in (0, 1)$,
 \begin{equation*}
U(t-\tau\theta, t-\tau)\mathbb{B}_{t-\tau}(R_1)\subset \mathbb{B}_{t-\tau\theta}(R_1),\ \ \forall \tau\geq (1-\theta)^{-1} e(R_1),  \ t\in\r.
\end{equation*}
 Thus, it follows from  \eqref{3.6}-\eqref{3.8}  and the fact:  $B(t)\subset \mathbb{B}_t(R_2)$ for all $t\in \mathbb{R}$ that
\begin{equation}\label{3.9}
 \begin{split}
 & \mathrm{dist}_{X_t}\left(U(t, t-\tau)\mathbb{B}_{t-\tau}(R_1), E(t)\right) \\
 \leq\ & \mathrm{dist}_{X_t}\left(U(t, t-\tau\theta)U(t-\tau\theta, t-\tau)\mathbb{B}_{t-\tau}(R_1), U(t, t-\tau\theta)B(t-\tau\theta)\right)\\
& + \mathrm{dist}_{X_t}\left(U(t, t-\tau\theta)B(t-\tau\theta), E(t)\right)\\
\leq\ & L_1^{(\frac{\tau\theta}{T}+1)}\mathrm{dist}_{X_{t-\tau\theta}}\left(U(t-\tau\theta, t-\tau)\mathbb{B}_{t-\tau}(R_1), B(t-\tau\theta)\right)+ Ce^{-\beta\tau\theta}\\
\leq\ &C_0L_1e^{\left(\frac{\theta}{T}\ln L_1+\kappa\theta-\kappa\right)\tau}+ Ce^{-\beta\tau\theta}= (C+C_0L_1)e^{-\beta'\tau}
  \end{split}
\end{equation}
for all $\tau \geq \tau_2:=\max\left\{\theta^{-1}\tau_0, (1-\theta)^{-1} e(R_1)\right\}$ and $t\in\mathbb{R}$, where
\begin{equation*}
  \theta=\frac{T\kappa}{2(\ln L_1+ T\kappa)}\in (0, 1) \ \ \hbox{and}\ \ \beta'=\min\left\{\frac{\kappa}{2}, \frac{T\kappa \beta}{2(\ln L_1+ T\kappa)}\right\}>0.
\end{equation*}
Then, repeating the same argument as the proof of Corollary \ref{ea1} and using estimate \eqref{3.9}, we obtain  that for any uniformly bounded family $\cald=\{D(t)\}_{t\in\mathbb{R}}$,
\begin{equation*}
\mathrm{dist}_{X_t}\left(U(t,t-\tau)D(t-\tau), E(t)\right)\leq (C+C_0L_1)e^{\beta' e(R)}e^{-\beta'\tau}, \ \ \forall \tau\geq \tau_2+e(R), \ t\in\r.
\end{equation*}
This  completes the proof.
\end{proof}
\medskip

\begin{proof}[\textbf{Proof of Theorem \ref{ea}}] For clarity and  without loss of generality, we assume $T=1$. Since $\b=\{B(t)\}_{t\in\r}$ is a uniformly bounded family, we have
\begin{equation*}
B(t)\subset \mathbb B_t(R_0),  \ \  \forall t\in\r
\end{equation*}
for some positive constant $R_0$. Thus  $N_t\left(B(t), R_0\right)=1$ for all $t\in\r$, here $N_t(B, \e)$ denotes the cardinality of minimal covering of the set $B\subset X_t$ by its closed subsets of diameter $\leq 2\e$.

Conditions $(H_1)$-$(H_3)$ show  that
\begin{equation}\label{3.10}
U(m, n)B(n)\subset B(m), \ \ \forall m\geq n,
\end{equation}
and  for all $x, y\in B(n-1)$ and $n\in\mathbb Z$,
\begin{align}
& \|U(n, n-1)x- U(n, n-1)y\|_{X_n}\leq L_1\|x-y\|_{X_{n-1}},\label{3.11}\\
& \|K_n x-K_ny\|_Z\leq L\|x-y\|_{X_{n-1}},\label{3.12}\\
&\|U(n, n-1)x- U(n, n-1)y\|_{X_n}\leq \eta\|x-y\|_{X_{n-1}}+n_Z\left(K_n x-K_ny\right).\label{3.13}
\end{align}

Now  we   show the following formula  for all $n\in\mathbb Z$ by induction on $k\in\mathbb N$
\begin{equation}\label{3.14}
N_n(k):= N_n\left(U(n, n-k)B(n-k), (2\eta)^k R_0\right)\leq \left[m_Z\left(\frac{2L}{\eta}\right)\right]^k.
\end{equation}

Let, for $k=1$ and $n\in\mathbb Z$,
\begin{equation*}
n(1):=\aleph\left\{x_l\in B(n-1)\ |\ n_Z\left(K_n x_j-K_nx_i\right)>\eta R_0,\ j\neq i\right\},
\end{equation*}
where $\aleph\{\cdots\}$ denotes the maximal number of elements with the given properties.
And let
\begin{equation*}
\b(n-1)=K_n B(n-1)=\left\{K_n x\ |\ x\in B(n-1)\right\}.
\end{equation*}
Formula  \eqref{3.12} implies that
\begin{align*}&
\mathrm{diam}\left(\b(n-1); Z\right)\leq L\mathrm{diam}\left(B(n-1); X_{n-1}\right)\leq 2L R_0,\\
& \b(n-1)\subset \mathbb B_Z(y; 2LR_0)=\left\{z\in Z\ |\ \|z-y\|_Z\leq 2LR_0\right\} \ \ \hbox{for some}\ \ y\in Z.
\end{align*}
By the linearity and compactness of the seminorm $n_Z(\cdot)$,
\begin{equation}\label{3.15}
\begin{split}
n(1) &= \aleph\left\{z_l\in \b(n-1)\ |\ n_Z\left(z_j-z_i\right)>\eta R_0,\ j\neq i\right\}\\
     &\leq \aleph\left\{z_l\in \mathbb B_Z(y; 2LR_0)\ |\ n_Z\left(z_j-z_i\right)>\eta R_0,\ j\neq i\right\}\\
 &= \aleph\left\{z_l\in \mathbb B_Z\left(0; \frac{2L}{\eta}\right)\ |\ n_Z\left(z_j-z_i\right)>1,\ j\neq i\right\}\\
 &=m_Z\left( \frac{2L}{\eta}\right)<+\infty.
\end{split}
\end{equation}
Consequently, there exists a maximal subset $\{x_j\}^{n(1)}_{j=1}$ of $B(n-1)$ such that
\begin{equation*}
n_Z\left(K_n x_j-K_nx_l\right)>\eta R_0, \ \ j \neq l,
\end{equation*}
and
\begin{align*}&
B(n-1)=\bigcup^{n(1)}_{j=1}C_j \ \ \hbox{with}\ \ C_j=\left\{x\in B(n-1)\ |\ n_Z\left(K_n x-K_nx_j\right)\leq \eta R_0\right\},\\
&
U(n, n-1)B(n-1)=\bigcup^{n(1)}_{j=1}U(n, n-1)C_j.
\end{align*}
Formula  \eqref{3.13} implies that, for all $x, y\in C_j$ and all $j=1, \cdots, n(1)$,
\begin{equation*}
\begin{split}
&\|U(n, n-1)x- U(n, n-1)y\|_{X_n}\\
\leq\ & \eta\|x-y\|_{X_{n-1}}+n_Z\left(K_n x-K_ny\right)\\
\leq\ & \eta \mathrm{diam}\left(B(n-1); X_{n-1}\right)+2\eta R_0\leq 4\eta R_0,
\end{split}
\end{equation*}
 which implies
\begin{equation}\label{3.16}
\mathrm{diam}\left(U(n, n-1)C_j; X_n\right)\leq 4\eta R_0, \ \ j=1, \cdots, n(1).
\end{equation}
Thus, the combination of \eqref{3.15} and \eqref{3.16} shows that
\begin{equation*}
N_n(1)= N_n\left(U(n, n-1)B(n-1), 2\eta R_0\right)\leq m_Z\left(\frac{2L}{\eta}\right),
\end{equation*}
that is, formula \eqref{3.14} holds for $k=1$ and $n\in\mathbb Z$.

Assume that formula \eqref{3.14} holds for all $1\leq k\leq k_0$ and $n\in\mathbb Z$.  We prove  that it also holds for $k=k_0+1$.  Due to
\begin{equation*}
U(n, n-k_0-1)B(n-k_0-1)= U(n, n-1)U(n-1, n-k_0-1)B(n-k_0-1),
\end{equation*}
and
\begin{equation}\label{3.17}
N_n(k_0)= N_n\left(U(n, n-k_0)B(n-k_0), (2\eta)^{k_0} R_0\right)\leq \left[m_Z\left(\frac{2L}{\eta}\right)\right]^{k_0}, \ \ \forall n\in\mathbb Z,
\end{equation}
 there exists a minimal covering $\{F_i\}^{N_{n-1}(k_0)}_{i=1}$ of $U(n-1, n-k_0-1)B(n-k_0-1)$ by its closed subsets of diameter $\leq 2(2\eta)^{k_0} R_0$, that is,
\begin{align}
& \bigcup^{N_{n-1}(k_0)}_{i=1} F_i= U(n-1, n-k_0-1)B(n-k_0-1)\subset B(n-1), \ \ \label{3.18}\\
& \mathrm{diam}\left(F_i; X_{n-1}\right)\leq 2(2\eta)^{k_0} R_0, \ \ i=1, \cdots, N_{n-1}(k_0).\label{3.19}
\end{align}
 Let
\begin{equation*}
\mathcal F_i= K_nF_i =\{K_n z\ |\ z\in F_i\}\subset Z,\ \ i=1, \cdots, N_{n-1}(k_0).
\end{equation*}
The combination of  \eqref{3.12} and \eqref{3.19} shows that
\begin{align*}&
\mathrm{diam}\left(\mathcal F_i; Z\right)\leq L\mathrm{diam}\left(F_i; X_{n-1}\right)\leq 2L(2\eta)^{k_0} R_0,\\
&\mathcal F_i\subset \mathbb B_Z\left(y_i; 2L(2\eta)^{k_0} R_0\right)=\left\{z\in Z\ |\ \|z-y_i\|_Z\leq 2L(2\eta)^{k_0} R_0\right\}\ \ \hbox{for some}\ \ y_i\in Z.
\end{align*}
 Hence,
\begin{equation}\label{3.20}
\begin{split}
n_i(k_0+1):&= \aleph\left\{x_l\in F_i\ |\ n_Z\left(K_n x_j-K_n x_m\right)>\eta (2\eta)^{k_0} R_0,\ j\neq m\right\}\\
 &= \aleph\left\{z_l\in \mathcal F_i\ |\ n_Z\left(z_j-z_m\right)>\eta (2\eta)^{k_0} R_0,\ j\neq m\right\}\\
     &\leq \aleph\left\{z_l\in \mathbb B_Z\left(y_i; 2L(2\eta)^{k_0} R_0\right)\ |\ n_Z\left(z_j-z_m\right)>\eta (2\eta)^{k_0} R_0,\ j\neq m\right\}\\
 &= \aleph\left\{z_l\in \mathbb B_Z\left(0; \frac{2L}{\eta}\right)\ |\ n_Z\left(z_j-z_m\right)>1,\ j\neq m\right\}\\
 &=m_Z\left( \frac{2L}{\eta}\right)<+\infty.
\end{split}
\end{equation}
Consequently,
\begin{equation}\label{3.21}
F_i=\bigcup^{n_i(k_0+1)}_{j=1}C^i_j \ \ \hbox{with}\ \ C^i_j=\left\{x\in F_i\ |\ n_Z\left(K_n x-K_nx^i_j\right)\leq \eta (2\eta)^{k_0} R_0\right\},
\end{equation}
where $\{x^i_j\}^{n_i(k_0+1)}_{j=1}$ is the maximal subset of $F_i$ such that
 \begin{equation*}
n_Z\left(K_n x^i_l-K_nx^i_j\right)> \eta (2\eta)^{k_0} R_0, \ \  l \neq j.
\end{equation*}
It follows from \eqref{3.13} and \eqref{3.17}-\eqref{3.21} that
\begin{align*}
& \mathrm{diam}\left(U(n, n-1)C_j; X_n\right)\leq \eta\mathrm{diam}\left(F_i; X_{n-1}\right)+(2\eta)^{k_0+1} R_0 \leq 2(2\eta)^{k_0+1} R_0,\\
& U(n, n-k_0-1)B(n-k_0-1)=\bigcup_{i=1}^{N_{n-1}(k_0)}\bigcup_{j=1}^{n_i(k_0+1)} U(n, n-1) C^i_j,
\end{align*}
which imply that
\begin{equation*}
\begin{split}
N_n(k_0+1)&=N_n\left( U(n, n-k_0-1)B(n-k_0-1),  (2\eta)^{k_0+1} R_0\right)\\
          &\leq \sum_{i=1}^{N_{n-1}(k_0)}n_i(k_0+1)\leq \left[m_Z\left(\frac{2L}{\eta}\right)\right]^{k_0+1}.
\end{split}
\end{equation*}
Therefore, formula \eqref{3.14} is valid.
\medskip

It follows from \eqref{3.10} and \eqref{3.14} that, for every $n\in\mathbb Z$ and every $k\geq 1$, there is a finite subset $V_k(n)$ possessing the following properties:
\begin{align}
& \mathrm{Card}\left(V_k(n)\right)\leq \left[m_Z\left(\frac{2L}{\eta}\right)\right]^k,\label{3.24}\\
& V_k(n)\subset U(n, n-k)B(n-k)\subset B(n)\subset X_n,\label{3.22}\\
& U(n, n-k)B(n-k)\subset \bigcup_{h\in V_k(n)} \left[\mathbb B_n\left(2(2\eta)^k R_0 \right)+\{h\}\right]. \label{3.23}
\end{align}
For any given $n\in\mathbb Z$, we define by induction that
\begin{equation}\label{3.25}
\left\{
  \begin{array}{ll}
    E_1(n)=V_1(n),  \\
~\\
    E_k(n)= V_k(n)\cup \left[U(n, n-1)E_{k-1}(n-1)\right],\ \ k\geq 2,  \\
~\\
    E(n)=\left[\bigcup_{k\geq 1}E_k(n)\right]_{X_n},
  \end{array}
\right.
\end{equation}
where  $[\cdot]_{X_n}$ denotes the closure in $X_n$. Thus  it follows from \eqref{3.10} and \eqref{3.22}-\eqref{3.23} that
\begin{align}
& E_k(n)=\bigcup_{l=0}^{k-1}U(n, n-l)V_{k-l}(n-l)\subset U(n, n-k)B(n-k),\label{3.26}\\
& U(n+1, n)E_k(n)\subset E_{k+1}(n+1),\label{3.27}\\
& E(n)\subset U(n, n-1)B(n-1)\subset B(n),\ \ \forall  n\in\mathbb Z, \ k\geq 1.\label{3.28}
\end{align}
  Moreover, we infer from \eqref{3.24} and \eqref{3.26} that
\begin{align}\label{3.29}
  \mathrm{Card}\left(E_k(n)\right)& \leq \sum_{l=0}^{k-1} \mathrm{Card}\left(V_{k-l}(n-l)\right)\nonumber\\
  &\leq \sum_{l=0}^{k-1}\left[m_Z\left(\frac{2L}{\eta}\right)\right]^{k-l}\leq \left[m_Z\left(\frac{2L}{\eta}\right)\right]^{k+1}, \ \ \forall  n\in\mathbb Z, \ k\geq 1.
\end{align}
 \medskip

We show that the family $\{E(n)\}_{n\in\mathbb Z}$ is of the following properties:

(i) Semi-invariance.  By the Lipschitz continuity \eqref{3.11}, formulas \eqref{3.25} and \eqref{3.27}-\eqref{3.28},
\begin{equation}\label{3.291}
U(n, l)E(l)\subset \Big[\bigcup_{k\geq 1}U(n, l)E_k(l)\Big]_{X_n}\subset\Big[\bigcup_{k\geq 1}E_{k+n-l}(n)\Big]_{X_n}\subset E(n), \ \ \forall n\geq l.
\end{equation}

(ii) Pullback exponential attractiveness. We see from \eqref{3.25} that $V_k(n)\subset E(n)$ holds for all $n\in\mathbb Z$ and $k\geq 1$. Thus  we infer from \eqref{3.23} that
\begin{equation}\label{3.30}
\begin{split}
& \mathrm{dist}_{X_n}\left(U(n, n-k)B(n-k), E(n)\right)\\
\leq \ & \mathrm{dist}_{X_n}\left(U(n, n-k)B(n-k), V_k(n)\right)\\
\leq\ & 2(2\eta)^k R_0, \ \ \forall k\geq 1, \ n\in\mathbb Z.
\end{split}
\end{equation}

(iii) Boundedness of the fractal dimension. For any $\e\in(0, 1)$, there is a unique $k_\e\in \mathbb N^+$ such that
\begin{equation}\label{3.31}
2(2\eta)^{k_\e} R_0<\e\leq 2(2\eta)^{k_\e-1} R_0.
\end{equation}
Obviously, $k_\e\rightarrow \infty$ as $\e\rightarrow 0^+$. It follows  from \eqref{3.10} and \eqref{3.26} that
\begin{equation*}
E_k(n)\subset U(n, n-k)B(n-k)\subset U(n, n-k_\e)B(n-k_\e), \ \ \forall k\geq k_\e, \ n\in\mathbb Z,
\end{equation*}
which implies that
\begin{equation*}
E(n)\subset \Big(\bigcup_{k<k_\e} E_k(n)\Big) \bigcup\Big[U(n, n-k_\e)B(n-k_\e) \Big]_{X_n}, \ \ \forall n\in\mathbb Z.
\end{equation*}
Thus it follows from \eqref{3.14}, \eqref{3.29} and \eqref{3.31} that
\begin{equation}\label{3.32}
\begin{split}
N_n\left(E(n), \e\right)&\leq N_n\left(E(n), 2(2\eta)^{k_\e} R_0\right)\\
 &\leq \sum_{k<k_\e} \mathrm{Card}\left(E_k(n)\right)+ N_n\left(U(n, n-k_\e)B(n-k_\e), 2(2\eta)^{k_\e} R_0\right)\\
&\leq \sum_{k<k_\e} \left[m_Z\left(\frac{2L}{\eta}\right)\right]^{k+1}+ N_n(k_\e)\leq 2\left[m_Z\left(\frac{2L}{\eta}\right)\right]^{k_\e+1}<+\infty.
\end{split}
\end{equation}
By the arbitrariness of $\e\in(0, 1)$, we see from \eqref{3.32} that $E(n)$ is a compact subset of $X_n$. Moreover, estimates \eqref{3.31}-\eqref{3.32} and a simple calculation shows that
\begin{equation*}
  \frac{\ln N_n\left(E(n),\epsilon\right)}{\ln{(1/\e)}}\leq \frac{(k_\e+1)\ln\left[m_Z\left(\frac{2L}{\eta}\right)\right]+\ln 2 }{(k_\e-1)\ln(1/ 2\eta)-\ln(2R_0)}, \ \ \e\in(0, 1),
\end{equation*}
which implies that
\begin{equation}\label{3.33}
  \mathrm{dim}_f\left(E(n); X_n \right)=\limsup_{\e\rightarrow 0^+}\frac{\ln N_n\left(E(n),\epsilon\right)}{\ln{(1/\e)}} \leq  \Big[\ln\Big(\frac{1}{2\eta}\Big)\Big]^{-1}\ln \left[ m_Z\left(\frac{2L}{\eta}\right)\right],\ \ \forall n\in\mathbb Z.
\end{equation}
\bigskip

For any $t\in\r$, there exists a unique $n\in\mathbb Z$ such that $t\in[n, n+1)$. Let
\begin{equation}\label{3.34}
E(t)=U(t, n)E(n).
\end{equation}

We claim that   $\{E(t)\}_{t\in\r}$ is the desired family.

 (i) It follows from formulas \eqref{3.1} and \eqref{3.28} that
\begin{equation*}
E(t)=U(t, n)E(n)\subset U(t, n)U(n, n-1)B(n-1)\subset B(t), \ \ \forall t\in\r.
\end{equation*}
Moreover, by the Lipschitz continuity  of $U(t, n): B(n)\subset X_n\rightarrow X_t $ (see \eqref{3.2}), and the compactness of $E(n)$ in $X_n$, we know that $E(t)$ is a compact subset of $X_t$ and by \eqref{3.33},
\begin{equation*}
\mathrm{dim}_f\left(E(t); X_t \right)\leq  \mathrm{dim}_f\left(E(n); X_n \right) \leq  \Big[\ln\Big(\frac{1}{2\eta}\Big)\Big]^{-1}\ln \left[ m_Z\left(\frac{2L}{\eta}\right)\right].
\end{equation*}

(ii) For any $t\geq s\in\r$, let $t=n+t_1$, $s=m+s_1$ for some $n, m\in\mathbb Z$ and $t_1, s_1\in [0, 1)$. Then, by formula \eqref{3.34} and the semi-invariance  of $\{E(n)\}_{n\in\mathbb Z}$ (see \eqref{3.291}), we have
\begin{equation*}
\begin{split}
U(t, s)E(s)&= U(t, n)U(n, s)U(s, m)E(m)\\
 &= U(t, n)U(n, m)E(m)\subset U(t, n)E(n)=E(t).
\end{split}
\end{equation*}

(iii) For any given $t\in\r$ and $\tau\geq 3$, there exist $n\in\mathbb Z$ and $k_\tau\in\mathbb N^+$ such that
\begin{equation*}
t\in[n ,n+1)\ \ \hbox{and}\ \ \tau\in [k_\tau+2, k_\tau+3),
\end{equation*}
which imply that
\begin{equation}\label{3.35}
n-k_\tau-(t-\tau)\geq 1, \ \ -k_\tau<-\tau+3,
\end{equation}
and by formula \eqref{3.1},
\begin{equation}\label{3.36}
\begin{split}
U(t, t-\tau)B(t-\tau)&= U(t, n)U(n, n-k_\tau)U(n-k_\tau, t-\tau)B(t-\tau)\\
&\subset U(t, n)U(n, n-k_\tau)B(n-k_\tau).
\end{split}
\end{equation}
 Thus we  infer from the Lipschitz continuity \eqref{3.2}, estimate \eqref{3.30} and formulas \eqref{3.34}-\eqref{3.35} that
\begin{equation*}
\begin{split}
& \mathrm{dist}_{X_t}\left( U(t, t-\tau)B(t-\tau), E(t)\right)\\
\leq\ &  \mathrm{dist}_{X_t}\left( U(t, n)U(n, n-k_\tau)B(n-k_\tau),  U(t, n)E(n)\right)\\
\leq\  & L_1\mathrm{dist}_{X_n}\left( U(n, n-k_\tau)B(n-k_\tau),  E(n)\right)\\
\leq\ & 2L_1 (2\eta)^{k_\tau}R_0= 2L_1R_0e^{-\beta k_\tau}\leq Ce^{-\beta\tau},\ \ \forall t\in\r, \ \tau\geq 3,
\end{split}
\end{equation*}
with $\beta=\ln \frac{1}{2\eta}$ and $C= 2L_1R_0e^3$. This  completes the proof.
\end{proof}

\begin{remark}
 Theorem \ref{ea} and  its Corollaries \ref{ea1} and \ref{ea2} are established, for simplicity, in a Banach space framework because of the definition of phase space $\h_t$ in the time-dependent memory kernel problem. However, they are  still  valid if the family of Banach spaces $\{X_t\}_{t\in \r}$ there is  replaced by  a family of normed linear spaces $\{X_t\}_{t\in \r}$.
\end{remark}

\section{Preliminaries and main results on the model \eqref{1.1}-\eqref{1.3}}

For any $\sigma\in\r$, we define the compactly nested Hilbert spaces
\begin{equation*}
H^\sigma=  D (A^{\sigma/2})
\end{equation*}
endowed with the inner products and the norms:
\begin{equation*}
\langle u, v\rangle_\sigma =\langle A^{\sigma/2}u, A^{\sigma/2} v\rangle_{L^2}, \ \   \|u\|_\sigma= \|A^{\sigma/2}u\|_{L^2},
\end{equation*}
respectively, where and in the context the operator $A$ is as shown in Eq. \eqref{1.1}, and $L^2=L^2(\Omega)$.  The index $\sigma$ will be  omitted whenever zero. The symbol $\langle\cdot,\cdot\rangle$ for the $L^2$-inner product will  also be used for the duality pairing between the  dual spaces. We denote
\begin{equation*}
L^p=L^p(\Omega), \ \  H=L^2, \ \ H^1=H^1_0(\Omega),  \ \ H^{-1}=H^{-1}(\Omega),\ \ H^2=H^2(\Omega)\cap H^1_0(\Omega),
\end{equation*}
with $p\geq 1$.
For every fixed time $t$ and index $\sigma$, we introduce the weighted $L^2$-spaces, hereafter they are  called memory spaces,
\begin{equation*}
\mathcal M^\sigma_t= L^2_{\mu_t}\left(\r^+; H^{\sigma+1}\right)=\left\{\xi: \r^+\rightarrow H^{\sigma+1}\ |\ \int_0^\infty \mu_t(s)\|\xi(s)\|^2_{\sigma+1} \d s<\infty \right\}
\end{equation*}
equipped with the weighted $L^2$-inner products
\begin{equation*}
\langle \eta, \xi\rangle_{\mathcal M^\sigma_t}=\int_0^\infty \mu_t(s)\langle \eta(s), \xi(s)\rangle_{\sigma+1}\d s.
\end{equation*}
And we define the extended memory spaces
 \begin{equation*}
\h_t^\sigma=H^{\sigma+1}\times H^\sigma\times \mathcal M^\sigma_t
\end{equation*}
equipped with the usual product norm
\begin{equation*}
\|(u, v, \eta)\|^2_{\h^\sigma_t}=\|u\|^2_{\sigma+1}+\|v\|^2_{\sigma}+\|\eta\|^2_{\mathcal M^\sigma_t}.
\end{equation*}
For any $r>0$, we   denote by
\begin{equation*}
\mathbb B^\sigma_t(r)=\left\{z\in \h^\sigma_t\ |\ \|z\|_{\h^\sigma_t}\leq r\right\}
\end{equation*}
the closed $r$-ball centered at zero of $\h_t^\sigma$.
\medskip

\subsection{Assumptions and well-posedness}

\begin{assumption} \cite{Pata2}\label{22} (i)\ Let $g\in H$ be independent of time, and let $f\in C^2(\r)$ with $f(0)=0$,
\begin{equation}\label{2.1}
|f''(u)|\leq c(1+|u|)\ \ \hbox{and}\  \ \liminf_{|u|\rightarrow \infty} f'(u)>-\lambda_1,
\end{equation}
for some $c\geq 0$,
where $\lambda_1>0$ is the first eigenvalue of $A$.
\medskip

(ii)\ The map $(t, s)\mapsto \mu_t(s): \r\times \r^+ \rightarrow \r^+$ possesses the following properties:
\begin{description}
  \item[\textbf{($M_1$)}] For every fixed $t\in\r$, the map $s \mapsto \mu_t(s)$ is nonincreasing, absolutely continuous and summable. We denote the total mass of $\mu_t$ by $
    \kappa(t)=\int_0^\infty \mu_t(s)\d s$.
   \item[\textbf{($M_2$)}] For every $\tau\in\r$, there exists a function $K_\tau : [\tau, \infty) \rightarrow \r^+$, summable on any interval $[\tau, T]$, such that $
       \mu_t(s)\leq K_\tau(t)\mu_\tau(s)$ for every $t\geq \tau$ and every $s>0$.
     \item[\textbf{($M_3$)}] For almost every fixed $s>0$, the map $t\mapsto \mu_t(s)$ is differentiable for all $t\in\r$, and $(t, s)\mapsto \mu_t(s) \in L^\infty (\mathcal K)$, $(t, s)\mapsto \partial_t\mu_t(s) \in L^\infty (\mathcal K)$
for every compact set $\mathcal K\subset \r\times \r^+$.
       \item[\textbf{($M_4$)}] There exists a $\delta>0$ such that $\partial_t\mu_t(s)+ \partial_s\mu_t(s)+\delta \kappa(t) \mu_t(s)\leq 0$  for every $t\in\r$ and almost every $s>0$.
   \item[\textbf{($M_5$)}] The function $t\rightarrow \kappa(t)$ fulfills: $
            \inf_{t\in\r} \kappa(t)>0$.
   \item[\textbf{($M_6$)}] The function $t\mapsto \partial_t\mu_t(s)$ satisfies the uniform integral estimate: $
       \sup_{t\in\r}\frac{1}{[\kappa(t)]^2}\int_0^\infty |\partial_t\mu_t(s)|\d s<\infty$.
   \item[\textbf{($M_7$)}]  For every $t\in\r$, the function $s\mapsto \mu_t(s)$ is bounded about zero, with $\sup_{t\in\r} \frac{\mu_t(0)}{[\kappa(t)]^2}<\infty$.
     \item[\textbf{($M_8$)}] For every $a<b\in\r$, there exists a $\nu>0$ such that $
          \int_\nu^{1/\nu} \mu_t(s)\d s\geq \frac{\kappa(t)}{2}$ for every $t\in [a, b]$.
\end{description}
\end{assumption}

\begin{remark}\label{remark21}
(i) The conditions $(M_1)$-$(M_8)$ are  quoted from \cite{Pata2}.  As is shown in  \cite{Pata2}, the following function
\begin{equation*}
\mu_t(s)=\frac{1}{[\varepsilon (t)]^2} e^{-\frac{s}{\varepsilon(t)}}\ \ \hbox{with}\ \ \varepsilon(t)=\frac{1}{4} \left[\frac{\pi}{2}-\arctan(t)\right]
\end{equation*}
satisfies above mentioned conditions $(M_1)$-$(M_8)$.

(ii)  Condition $(M_2)$ implies the continuous embedding: $ \mathcal M^\sigma_\tau \hookrightarrow \mathcal M^\sigma_t$ for all $\sigma\in\r$ and $t>\tau$, with
   \begin{equation*}
    \|\eta\|^2_{\mathcal M^\sigma_t}\leq K_\tau(t) \|\eta\|^2_{\mathcal M^\sigma_\tau}, \ \ \forall \eta\in \mathcal M^\sigma_\tau.
   \end{equation*}
Therefore, $ \h^\sigma_\tau \hookrightarrow \h^\sigma_t$ for all $\sigma\in\r$ and $t>\tau$.
\end{remark}

Now, we quote some   known results in recent literatures  \cite{Pata1,Pata2}, which are the bases of  our arguments. Let us begin with the definition of weak solution.

\begin{definition}\label{weaksolotion}\cite{Pata1} Let $T>\tau\in\r$, and  $z_\tau=\left(u_\tau, v_\tau, \eta_\tau\right)\in \h_\tau$ be a fixed vector. A function
 \begin{equation*}
  z(t)=\left(u(t), \partial_t u(t), \eta^t\right)\in \h_t \ \  \hbox{for a.e.}\ \ t\in [\tau, T]
 \end{equation*}
is called a weak solution of problem \eqref{1.1}-\eqref{1.3} on  interval $[\tau, T]$ if $u(\tau)=u_\tau$, $\partial_t u(\tau)=v_\tau$ and
  \begin{enumerate}[$(i)$]
    \item $u\in L^\infty \left(\tau, T; H^1\right)$, $\partial_t u\in L^\infty \left(\tau, T; H\right)$, $\partial_{tt} u\in L^1 \left(\tau, T; H^{-1}\right)$;
    \item the function $\eta^t$ fulfills the representation formula \eqref{1.2};
    \item the function  $u(t)$ fulfills \eqref{1.1} in the weak sense, i.e.,
          \begin{equation*}
          \langle \partial_{tt} u(t), \phi\rangle + \langle u(t), \phi\rangle_1
         + \int_0^\infty \mu_t(s)\langle \eta^t(s), \phi\rangle_1 \d s + \langle f\left(u(t)\right), \phi\rangle=\langle g, \phi\rangle
           \end{equation*}
    for almost every $t\in [\tau, T]$ and every   $\phi\in H^1$.
  \end{enumerate}
\end{definition}

\begin{theorem}\label{existence} \cite{Pata1} Let Assumption  \ref{22} be valid. Then  for every $T>\tau\in\r$, and every   $z_\tau=\left(u_\tau, v_\tau, \eta_\tau\right)\in \h_\tau$,  problem \eqref{1.1}-\eqref{1.3} admits a unique weak solution
$z(t)=\left(u(t), \partial_t u(t), \eta^t\right)$ on $[\tau, T]$, with
\begin{equation*}
 (u, \partial_t u)\in C\left([\tau, T]; H^1\times H\right), \ \eta^t\in\mathcal M_t, \ \ \forall t\in [\tau, T].
\end{equation*}
Moreover, for any two  weak solutions $z_1(t)$ and $z_2(t)$ on $[\tau, T]$ with $\|z_1(\tau)\|_{\h_\tau}+ \|z_2(\tau)\|_{\h_\tau}\leq R$,
\begin{equation}\label{2.6}
\|z_1(t)-z_2(t)\|_{\h_t}\leq \mathcal Q(R) e^{\mathcal Q(R)( t-\tau)} \|z_1(\tau)-z_2(\tau)\|_{\h_\tau}, \ \ \forall t\in [\tau, T],
\end{equation}
where $\mathcal Q$ is an increasing positive function independent of $t$.
\end{theorem}

\subsection{Main results }

Under Assumption  \ref{22}, we define the mapping
\begin{equation*}
   U(t, \tau): \h_\tau \rightarrow \h_t, \ \   U(t, \tau) z_\tau=z(t), \ \ \forall z_\tau\in \h_\tau, \   t\geq \tau,
   \end{equation*}
where  $z(t)$ is the weak solution of problem \eqref{1.1}-\eqref{1.3} corresponding to the initial data $z_\tau\in\h_\tau$. By Theorem \ref{existence}, the two-parameter family $\{U(t, \tau)\ |\ t\geq \tau\}$ constitutes  a process acting on time-dependent Banach spaces $\{\h_t\}_{t\in\r}$.

 \begin{lemma}\label{gaexist} \cite{Pata2} Let Assumption  \ref{22} be valid. Then  the process  $U(t, \tau): \h_\tau \rightarrow \h_t$ generated by problem \eqref{1.1}-\eqref{1.3} possesses the invariant time-dependent global attractor $\a=\{A(t)\}_{t\in\r}$, with $A(t)\subset \h_t^1$ for each $t\in\r$, and
\begin{equation*}
\sup_{t\in\r}\|A(t)\|_{\h^1_t}<\infty.
\end{equation*}
\end{lemma}
\medskip

Now, we state the main results of the present paper, and its proof will be given in Section 5 after some delicate technique  preparations in Section 4.

\begin{theorem}\label{eaexist}Let Assumption \ref{22} be valid. Then  the process  $U(t, \tau): \h_\tau \rightarrow \h_t$ generated by problem \eqref{1.1}-\eqref{1.3} admits a time-dependent exponential attractor $\E=\{E(t)\}_{t\in\r}$, with $E(t)\subset \h_t^1$ for each $t\in\r$, and
\begin{equation*}
\sup_{t\in\r}\|E(t)\|_{\h^1_t}<\infty.
\end{equation*}
\end{theorem}

\begin{corollary} Let   Assumption \ref{22} be valid. Then the fractal dimension of the  invariant time-dependent global attractor $\a=\{A(t)\}_{t\in\r}$ given by Lemma \ref{gaexist} is uniformly bounded, that is,
\begin{equation*}
  \sup_{t\in\r} \dim_f\left( A(t); \h_t\right)\leq     \sup_{t\in\r} \dim_f\left( E(t); \h_t\right) <+\infty.
    \end{equation*}
\end{corollary}

\section{Some key estimates}

We first quote some known results (Lemmas \ref{41} to \ref{48}) coming from literature \cite{Pata1,Pata2}, which will be the stating point of our argument.

\begin{lemma}\label{41}\cite{Pata2}  (Gronwall-type lemma in integral form)  Let $\tau\in \r$ be fixed, and   $\Lambda: [\tau, \infty)\rightarrow \r$ be a continuous function. Assume that,  for some $\e>0$ and every $b>a\geq \tau$,
\begin{equation*}
\Lambda(b)+2\e \int_a^b \Lambda(y)\d y\leq \Lambda(a)+\int_a^b q_1(y)\Lambda(y)\d y +\int_a^b q_2(y) \d y,
\end{equation*}
  where $q_1, q_2$ are locally summable nonnegative functions on $[\tau, \infty)$ satisfying
\begin{equation*}
\int_a^b q_1(y)\d y\leq \e(b-a)+c_1\ \ \hbox{and}\ \ \sup_{t\geq \tau}\int_t^{t+1} q_2(y) \d y\leq c_2,
\end{equation*}
for some $c_1, c_2\geq 0$. Then, we have
\begin{equation*}
\Lambda(t)\leq e^{c_1}\Big[\left|\Lambda(\tau)\right|e^{-\e(t-\tau)}+\frac{c_2e^\e}{1-e^{-\e}}\Big],\ \ \forall t\geq \tau.
\end{equation*}
\end{lemma}
\medskip

 We consider the following problem
\begin{align}
&\partial_{tt}p(t)+Ap(t)+\int_0^\infty\mu_t(s)A\psi^t(s)\d s+\gamma(t)=0,\ \   t>\tau,  \label{4.1}\\
& \psi^t(s)={\left\{
              \begin{array}{ll}
                p(t)-p(t-s), & s\leq t-\tau, \\
                \psi_\tau(s-t+\tau)+p(t)-p_\tau, & s>t-\tau,
              \end{array}
            \right.}\label{4.2}\\
&p(\tau)=p_\tau, \ \ \partial_t p(\tau)=q_\tau, \ \ \psi^\tau=\psi_\tau,\label{4.3}
\end{align}
where $\gamma$ is a certain forcing term (possibly depending on $p$) and  $\left(p_\tau, q_\tau, \psi_\tau\right)\in\h_\tau$. Assuming that problem \eqref{4.1}-\eqref{4.3} admits a global solution
$\left(p(t), \partial_t p(t), \psi^t\right)\in \h_t$ for all $t\in [\tau, \infty)$.

\begin{lemma}\label{42}\cite{Pata1} Let Assumption \ref{22}: $(ii)$ be valid. For any fixed $\sigma\in \r$ and every $T>\tau\in\r$, if also
\begin{equation*}
p\in W^{1, \infty}\left(\tau, T; H^{\sigma+1}\right)\ \  \hbox{and}\ \ \psi_\tau\in \mathcal M ^\sigma_\tau,
\end{equation*}
then for all $\tau\leq a\leq b\leq T$,
\begin{equation*}
\|\psi^b\|^2_{\mathcal M^\sigma_b}-\int_a^b\int_0^\infty\left[\partial_t \mu_t(s)+ \partial_s\mu_t(s)\right]\|\psi^t(s)\|^2_{\sigma+1}\d s\d t\leq \|\psi^a\|^2_{\mathcal M^\sigma_a} + 2\int_a^b \langle \partial_t p(t), \psi^t\rangle_{\mathcal M^\sigma_t}\d t.
\end{equation*}
\end{lemma}

\begin{lemma}\label{43}\cite{Pata2} Let Assumption \ref{22}: $(ii)$ be valid, and  the global solution $\left(p(t), \partial_t p(t), \psi^t\right)$ of problem \eqref{4.1}-\eqref{4.3} be sufficiently regular and let the functionals
\begin{align}
   &\Phi(t)=2\langle p(t), \partial_t p(t)\rangle,\label{4.6}\\
   & \Psi(t)= -\frac{2}{\kappa(t)}\int_0^\infty \mu_t(s)\langle \psi^t(s),  \partial_t p(t)\rangle\d s.\label{4.7}
\end{align}
Then, for every $\varpi\in (0, 1]$ and every $b>a\geq \tau$, we have
\begin{align}
\Phi(b)+(2-\varpi)\int_a^b\|p(t)\|^2_1\d t\leq\ & \Phi(a)+2 \int_a^b\|\partial_t p(t)\|^2\d t \label{4.4}\\
&\ +\frac{1}{\varpi}\int_a^b\kappa(t)\|\psi^t\|^2_{\mathcal M_t}\d t-2\int_a^b\langle \gamma(t), p(t)\rangle \d t,\nonumber
\end{align}
and
\begin{align}
\Psi(b)+\int_a^b\|\partial_tp(t)\|^2\d t\leq\ & \Psi(a)-M\int_a^b\int_0^\infty\left[\partial_t\mu_t(s)+ \partial_s\mu_t(s)\right]\|\psi^t(s)\|^2_1\d s\d t\nonumber\\
&\  +\varpi\int_a^b\|p(t)\|^2_1\d t +\frac{C}{\varpi}\int_a^b\kappa(t)\|\psi^t\|^2_{\mathcal M_t}\d t\label{4.5}\\
&\   +\int_a^b \frac{2}{\kappa(t)}\int_0^\infty\mu_t(s)\langle \psi^t(s), \gamma(t)\rangle \d s\d t,\nonumber
\end{align}
 where  $M$ and $C$ are positive constants depending only on the structural assumptions on the memory kernel.
\end{lemma}

\begin{remark}  (i) In the following in our arguments, Lemma \ref{43} is always used to the Galerkin approximations which are of enough regularity.

(ii)  By conditions $(M_1)$, $(M_4)$-$(M_5)$  one  easily sees that
\begin{equation}\label{4.8}
\left|\Phi(t)\right|+ \left|\Psi(t)\right|\leq C\|\left(p(t), \partial_t p(t), \psi^t\right)\|^2_{\h_t}, \ \ \forall t\geq \tau.
\end{equation}
\end{remark}
\medskip

\begin{lemma}\label{45}\cite{Pata2} Let Assumption \ref{22} be valid,  and $z_\tau\in \h_\tau$  with $\|z_\tau\|_{\h_\tau}\leq R$.  Then there exist positive constants $\omega $ and $R_0 $, which are independent of $R$,   such that
\begin{equation*}
\E(t, \tau):=\frac12\|U(t, \tau)z_\tau\|^2_{\h_t}\leq \mathcal Q(R) e^{-\omega(t-\tau)}+ R_0,\ \ \forall t\geq \tau.
\end{equation*}
That is,  the family  $\{\mathbb B_t(R_1)\}_{t\in\r}$, with $R_1>\sqrt{2R_0}$  is a   uniformly time-dependent absorbing set of the process $U(t, \tau)$.
\end{lemma}
\medskip

%\begin{corollary}\label{46} Under the assumptions of Theorem \ref{45}, we have
%\begin{equation*}
%\E(t, \tau)\leq \mathcal Q(R)\ \  \hbox{for all}\ \ t\geq \tau.
%\end{equation*}
%\end{corollary}
%\medskip

Following the standard method given in \cite{Pata3}, we split the nonlinearity $f$ into the sum
\begin{equation}\label{4.11'}
f(u)=f_0(u)+f_1(u),
\end{equation}
where $f_1\in C^2(\r)$ is globally Lipschitz with $f_1(0)=0$, while $f_0\in C^2(\r)$ vanishes inside $[-1, 1]$ and
\begin{align}
\left|f''_0(u)\right|\leq C|u|,\ \ f'_0(u)\geq 0 \label{4.11}
\end{align}
for some positive constant $C$. Then  we decompose the solution $U(t, \tau)z_\tau$ as the sum
\begin{equation*}
U(t, \tau)z_\tau=U_0(t, \tau)z_\tau+U_1(t, \tau)z_\tau,
\end{equation*}
where $U_0(t, \tau)z_\tau=\left(v(t), \partial_t v(t), \xi^t \right)$ solves the problem
\begin{equation}\label{4.13}
\left\{
  \begin{array}{ll}
    \partial_{tt}v(t)+Av(t)+\int_0^\infty\mu_t(s)A\xi^t(s)\d s+f_0\left(v(t)\right)=0, \\
    U_0(\tau, \tau)z_\tau=z_\tau,
  \end{array}
\right.
\end{equation}
where
\begin{equation*}
\xi^t(s)=\left\{
            \begin{array}{ll}
              v(t)-v(t-s), & s\leq t-\tau, \\
              \xi_\tau(s-t+\tau)+v(t)-v_\tau, & s>t-\tau,
            \end{array}
          \right.
\end{equation*}
and $U_1(t, \tau)z_\tau=\left(w(t), \partial_t w(t), \zeta^t \right)$ solves the problem
\begin{equation}\label{4.14}
\left\{
  \begin{array}{ll}
    \partial_{tt}w(t)+Aw(t)+\int_0^\infty\mu_t(s)A\zeta^t(s)\d s+f_0\left(u(t)\right)-f_0\left(v(t)\right)+ f_1\left(u(t)\right)=g, \\
    U_1(\tau, \tau)z_\tau=0,
  \end{array}
\right.
\end{equation}
where
\begin{equation*}
\zeta^t(s)=\left\{
            \begin{array}{ll}
              w(t)-w(t-s), & s\leq t-\tau, \\
              \zeta_\tau(s-t+\tau)+w(t)-w_\tau, & s>t-\tau.
            \end{array}
          \right.
\end{equation*}

\begin{lemma}\label{48} \cite{Pata2} Let Assumption \ref{22} be valid, and $z_\tau\in \h_\tau$ with $\|z_\tau\|_{\h_\tau}\leq R$.  Then for every $t\geq \tau$,
\begin{align}
& \|U_0(t, \tau)z_\tau\|^2_{\h_t}\leq \mathcal Q(R)e^{-\omega (t-\tau)},\label{4.15}\\
& \|U_1(t, \tau)z_\tau\|^2_{\h^{1/3}_t}\leq \mathcal Q(R),\label{4.16}
\end{align}
hereafter  $\omega>0$ is as shown in Lemma \ref{45}.
\end{lemma}

   Now, based on Lemmas \ref{41} to \ref{48}, we further establish some  new  estimates which will play key roles for our applying  Theorem \ref{ea} and its corollary to problem \eqref{1.1}-\eqref{1.3} to establish    the existence of the time-dependent exponential attractors.

\begin{lemma}\label{49} Let Assumption \ref{22} be valid, and   $z_\tau\in \h_\tau^{1/3}$ with $\|z_\tau\|_{\h_\tau}\leq R$. Then for every $t\geq \tau$, we have
\begin{equation*}
\|U(t, \tau)z_\tau\|^2_{\h_t^{1/3}}\leq \mathcal Q\left(R+\|z_\tau\|_{\h_\tau^{1/3}}\right)e^{-\omega (t-\tau)}+ \mathcal Q(R).
\end{equation*}
\end{lemma}

\begin{proof} For any  $t\geq \tau$, we define the functionals of the solutions  $(u(t), \partial_t u(t), \eta^t)= U(t, \tau)z_\tau$ as follows:
\begin{align}
&\E_{1/3}(t, \tau)= \frac12\|U(t, \tau)z_\tau\|^2_{\h^{1/3}_t},\nonumber\\
&\mathcal L_{1/3}(t) = L_{1/3}(t)+ \|\eta^t\|^2_{{\mathcal M}^{1/3}_t}, \label{4.18}\\
& \Lambda_{1/3}(t)= \mathcal L_{1/3}(t)+ 2\e \Big[\Phi(t)+ 4\Psi(t)\Big],\label{4.19}
\end{align}
where $\e\in (0, 1], \Phi$ and $\Psi$ are as shown    in \eqref{4.6}-\eqref{4.7}, with $
\left(p(t), \partial_t p(t), \psi^t\right)=  \left(A^{1/6}u(t), \partial_t A^{1/6}u(t), A^{1/6}\eta^t\right)$, and the functional
\begin{equation*}
L_{1/3}(t)=\|u(t)\|^2_{4/3}+ \|\partial_tu(t)\|^2_{1/3}+ 2\langle \gamma(t),  A^{1/6}u(t)\rangle,\ \ \hbox{with}\ \ \gamma(t)=A^{1/6} \left(f(u(t))-g\right).
\end{equation*}
By  condition \eqref{2.1} and  Lemma \ref{45}, we have
\begin{equation}\label{4.20}
\|f(u)\|\leq C\|1+|u|^3\|\leq C\left( 1+\|u\|^3_{L^6}\right)\leq C\left( 1+\|u\|^3_1\right)\leq \mathcal Q(R),
\end{equation}
which implies that
\begin{equation*}
2|\langle \gamma(t),  A^{1/6}u(t)|\rangle\leq 2\left[\|f(u)\|+\|g\|\right]\|u \|_{2/3}\leq \frac14 \|u(t)\|^2_{4/3}+\mathcal Q(R),
\end{equation*}
and hence,
\begin{equation}\label{4.21}
\frac32 \E_{1/3}(t, \tau)- \mathcal Q(R)\leq \mathcal L_{1/3}(t)\leq \frac52\E_{1/3}(t, \tau)+ \mathcal Q(R).
\end{equation}
By formula \eqref{4.8},
\begin{equation}\label{4.18'}
\left|\Phi(t)\right|+ \left|\Psi(t)\right|\leq C\E_{1/3}(t, \tau),\ \ \forall t\geq \tau.
\end{equation}
The combination of  \eqref{4.19} and \eqref{4.21}-\eqref{4.18'}  yields
\begin{equation}\label{4.26}
 \E_{1/3}(t, \tau)- \mathcal Q(R)\leq  \Lambda_{1/3}(t)\leq 3 \E_{1/3}(t, \tau)+ \mathcal Q(R)
\end{equation}
for $\e>0$ suitably small. Taking the multiplier $2A^{1/3} \partial_t u$ in Eq. \eqref{1.1} gives
\begin{equation}\label{4.22}
\frac{\d}{\d t}L_{1/3}(t)+ 2\langle \eta^t, \partial_tu(t)\rangle _{\mathcal M_t^{1/3}}= 2\left \langle f'(u(t))\partial_t u(t), A^{1/3}u(t)\right \rangle:=I_1(t)+I_2(t)+I_3(t),
\end{equation}
where
\begin{align*}&
I_1=2\left \langle \left[f'_0(u)- f'_0(v)\right]\partial_t u, A^{1/3}u\right \rangle,\\
&I_2=2\left \langle  f'_0(v)\partial_t u, A^{1/3}u\right \rangle,\\
&I_3= 2\left \langle f'_1(u)\partial_t u, A^{1/3}u\right \rangle.
\end{align*}
By Lemma  \ref{45}, estimates \eqref{4.11}, \eqref{4.15}-\eqref{4.16} and the Sobolev embedding
\begin{equation}\label{4.24}
H^1\hookrightarrow L^{6},\ \ H^{4/3}\hookrightarrow L^{18}, \ \  H^{2/3}\hookrightarrow L^{18/5}, \ \  H^{1/3}\hookrightarrow L^{18/7},
\end{equation}
we have
\begin{equation*}
\begin{split}
|I_1|&\leq C\int_\Omega \left(|u|+ |v|\right)|w||\partial_t u|| A^{1/3} u|\d x\\
&\leq C\left(\|u\|_{L^6 }+ \|v\|_{ L^6}\right)\|w\|_{L^{18} }\|\partial_t u\|\| A^{1/3} u\|_{L^{18/5}}\\
&\leq \alpha \|u\|^2_{4/3}+ \frac{\mathcal Q(R)}{\alpha},\ \ \forall \alpha\in (0, 1],\\
|I_2|&\leq C\int_\Omega |v|^2|\partial_t u|| A^{1/3} u|\d x\\
&\leq C \|v\|_{L^6 }^2\|\partial_t u\|_{L^{18/7}}\| A^{1/3} u\|_{L^{18/5}}\\
&\leq C\|v\|_1^2\left[ \|\partial_t u\|^2_{1/3}+ \|u\|^2_{4/3}\right],
\end{split}
\end{equation*}
and
\begin{equation*}
|I_3|\leq C\|\partial_t u\|\|A^{1/3}u\|\leq \mathcal Q(R) \|u\|_{4/3}\leq \alpha \|u\|^2_{4/3}+ \frac{\mathcal Q(R)}{\alpha},\ \ \forall \alpha\in (0, 1].
\end{equation*}
Inserting above estimates into \eqref{4.22} and making use of \eqref{4.21} receive
\begin{equation}\label{4.23}
\frac{\d}{\d t} L_{1/3}(t)+ 2\langle \eta^t, \partial_t u\rangle _{\mathcal M_t^{1/3}}
\leq q_1(t) \mathcal L_{1/3}(t) + \mathcal Q(R) q_1(t)+ \frac{\mathcal Q(R)}{\alpha},\ \ \forall \alpha\in (0, 1],
\end{equation}
 where $q_1(t)=C\left(\alpha+ \|v(t)\|_1^2\right)$.
   Integrating inequality \eqref{4.23} over $[a, b]$, with $b\geq a\geq \tau$, yields
\begin{equation*}
\begin{split}
& L_{1/3}(b)+ 2\int_a^b\langle \eta^t, \partial_t u(t)\rangle _{\mathcal M_t^{1/3}}\d t\\
\leq \ &L_{1/3}(a)+ \int_a^b q_1(t) \mathcal L_{1/3}(t)\d t + \int_a^b \left(q_1(t)+ \frac{\mathcal Q(R)}{\alpha}\right)\d t, \ \ \forall \alpha\in (0, 1].
\end{split}
\end{equation*}
It follows form Lemma \ref{42} (taking  $\sigma=1/3$ there) that
\begin{equation*}
\begin{split}
& \|\eta^b\|^2_{\mathcal M_b^{1/3}}-\int_a^b\int_0^\infty\left[\partial_t\mu_t(s)+ \partial_s \mu_t(s) \right]\|\eta^t(s)\|^2_{4/3}\d s\d t\\
\leq\ &\|\eta^a\|^2_{\mathcal M_a^{1/3}}+2\int_a^b\langle\partial_t u(t), \eta^t\rangle_{\mathcal M_t^{1/3}}\d t.
\end{split}
\end{equation*}
Adding above two   inequalities together we obtain
\begin{equation}\label{4.25}
\begin{split}
& \mathcal L_{1/3}(b)- \int_a^b\int_0^\infty\left[\partial_t\mu_t(s)+ \partial_s \mu_t(s) \right]\|\eta^t(s)\|^2_{4/3}\d s\d t\\
\leq\ &\mathcal L_{1/3}(a)+ \int_a^b q_1(t) \mathcal L_{1/3}(t)\d t + \int_a^b \left(q_1(t)+ \frac{\mathcal Q(R)}{\alpha}\right)\d t, \ \ \forall \alpha\in (0, 1].
\end{split}
\end{equation}

Exploiting estimates \eqref{4.4}-\eqref{4.5} ( taking $\varpi= \frac{1}{20}$ there), we get
\begin{equation}\label{4.27}
\begin{split}
& \Phi(b)+ 4\Psi(b)+\frac{7}{4}\int_a^b \| u(t)\|^2_{4/3}\d t+2\int_a^b \|\partial_t u(t)\|^2_{1/3}\d t\\
\leq\ & \Phi(a)+ 4\Psi(a)- 4 M\int_a^b\int_0^\infty\left[\partial_t\mu_t(s)+ \partial_s \mu_t(s) \right]\|\eta^t(s)\|^2_{4/3}\d s\d t\\
&+\  C\int_a^b\kappa(t)\|\eta^t\|^2_{\mathcal M_t^{1/3}}\d t -2\int_a^b \langle \gamma(t), A^{1/6}u(t)\rangle \d t\\
 &+\ 8\int_a^b \frac{1}{\kappa(t)}\int_0^\infty\mu_t(s)\langle A^{1/6}\eta^t(s), \gamma(t)\rangle \d s\d t.
\end{split}
\end{equation}
Due to
\begin{equation*}
\|\gamma(t)\|_{-1}=\|A^{-1/3}(f(u)-g)\| \leq C\left(\|f\left(u(t)\right)\|+\|g\|\right)\leq \mathcal Q(R),
\end{equation*}
we have
\begin{equation}\label{4.28}
\begin{split}
-2\int_a^b \langle \gamma(t), A^{1/6}u(t)\rangle \d t &\leq 2\int_a^b \|\gamma(t)\|_{-1}\|u(t)\|_{4/3} \d t\\
& \leq\frac {1}{4} \int_a^b \| u(t)\|^2_{4/3}\d t+ \mathcal Q(R)(b-a),
\end{split}
\end{equation}
and  by  conditions $(M_1)$ and $(M_5)$, we have
\begin{align}
&8\int_a^b \frac{1}{\kappa(t)}\int_0^\infty\mu_t(s)\langle A^{1/6}\eta^t(s), \gamma(t)\rangle \d s\d t\nonumber\\
\leq \  & 8\int_a^b \frac{1}{\kappa(t)}\|\gamma(t)\|_{-1} \left(\int_0^\infty\mu_t(s)\|\eta^t(s)\|_{4/3} \d s\right)\d t\nonumber\\
\leq \ & \mathcal Q(R) \int_a^b \frac{1}{\kappa(t)} \sqrt{\kappa(t)}\|\eta^t\|_{\mathcal M_t^{1/3}}\d t\label{4.29}\\
\leq \ &  \mathcal Q(R)(b-a)+ \mathcal Q(R) \int_a^b \kappa(t)\|\eta^t\|^2_{\mathcal M_t^{1/3}}\d t\nonumber.
\end{align}
Inserting estimates \eqref{4.28}-\eqref{4.29} into \eqref{4.27} and making use of condition $(M_5)$ and  \eqref{4.26} turn out
\begin{equation}\label{4.30}
\begin{split}
& \Phi(b)+ 4\Psi(b)+\int_a^b\Lambda_{1/3}(t)\d t\\
\leq\ & \Phi(a)+ 4\Psi(a)- 4M\int_a^b\int_0^\infty\left[\partial_t\mu_t(s)+ \partial_s \mu_t(s) \right]\|\eta^t(s)\|^2_{4/3}\d s\d t\\
& +\  \mathcal Q(R) \int_a^b \kappa(t)\|\eta^t\|^2_{\mathcal M_t^{1/3}}\d t + \mathcal Q(R)(b-a).
\end{split}
\end{equation}
The combination of \eqref{4.25}  and \eqref{4.30} yields
\begin{equation}\label{4.31}
\begin{split}
&  \Lambda_{1/3}(b)+2\e \int_a^b\Lambda_{1/3}(t)\d t+\mathcal J \\
\leq\ &  \Lambda_{1/3}(a)+ \int_a^b q_1(t) \mathcal L_{1/3}(t)\d t + \mathcal Q(R)\int_a^b \left(q_1(t)+\frac{1}{\alpha} \right)\d t \\
\leq\ &  \Lambda_{1/3}(a)+ \frac 52 \int_a^b q_1(t) \Lambda_{1/3}(t)\d t + \mathcal Q(R)\int_a^b \left(q_1(t)+\frac{1}{\alpha} \right)\d t,
\end{split}
\end{equation}
where
\begin{equation*}
\begin{split}
\mathcal J&= -(1-8\e M)\int_a^b\int_0^\infty\left[\partial_t\mu_t(s)+ \partial_s \mu_t(s) \right]\|\eta^t(s)\|^2_{4/3}\d s\d t-2\e \mathcal Q(R) \int_a^b \kappa(t)\|\eta^t\|^2_{\mathcal M_t^{1/3}}\d t\\
&\geq \left(\delta(1-8\e M)-2\e \mathcal Q(R)\right)\int_a^b \kappa(t)\|\eta^t\|^2_{\mathcal M_t^{1/3}}\d t\geq 0
\end{split}
\end{equation*}
for $\e\in (0, 1]$ suitably small. Taking $\alpha: C\alpha=\omega\leq \epsilon$, a simple calculation shows that
\begin{align*}
&\frac 52 \int_a^b  q_1(t)\d t\leq  C\alpha(b-a)+ \mathcal Q(R)\int_a^b e^{-\omega(t-\tau)} dt\leq
\omega(b-a)+ \mathcal Q(R),\\
&
 \sup_{t\geq \tau}\int_t^{t+1}Q(R) \left(q_1(s)+\frac{1}{\alpha} \right)\d s\leq \mathcal Q(R)\Big(\omega+\frac{Q(R)}{\omega}\Big)=Q(R).
\end{align*}
Therefore, applying Lemma \ref{41} to \eqref{4.31} and making use of \eqref{4.26} give the conclusion of  Lemma \ref{49}.
\end{proof}

For any fixed $\tau\in\r$ and $z_\tau\in \h_\tau^{1/3}$ with $\|z_\tau\|_{\h_\tau^{1/3}}\leq R$, we still write
\begin{equation}\label{fenjie}
U(t, \tau)z_\tau=U_0(t, \tau)z_\tau+U_1(t, \tau)z_\tau,
\end{equation}
where $U_0(t, \tau)$ and $U_1(t, \tau)$ are the solution operators of problems  \eqref{4.13} and \eqref{4.14},  respectively, with
\begin{equation*}
f_0(u)=0 \ \ \hbox{and}\ \ f_1(u)=f(u).
\end{equation*}

It follows from Lemma \ref{49} that, for every  $z_\tau\in \h_\tau^{1/3}$ with $\|z_\tau\|_{\h_\tau^{1/3}}\leq R$,
\begin{equation}\label{4.32}
\|U(t, \tau)z_\tau\|^2_{\h_t^{1/3}}\leq \mathcal Q(R), \ \ \forall t\geq \tau.
\end{equation}
Thus, by using the same argument as Lemma 8.1 in \cite{Pata2}, we have (the proof is omitted here)

\begin{lemma}\label{410}   Let Assumption \ref{22} be valid, and  $z_\tau\in \h_\tau^{1/3}$ with $\|z_\tau\|_{\h_\tau^{1/3}}\leq R$. Then  for every $t\geq \tau$,
\begin{align}
&\|U_0(t, \tau)z_\tau\|^2_{\h_t}\leq \mathcal Q(R) e^{-\omega (t-\tau)},\label{4.33}\\
& \|U_1(t, \tau)z_\tau\|^2_{\h^1_t}\leq \mathcal Q(R).\label{4.34}
\end{align}
\end{lemma}

Based on Lemma \ref{410}, we give a further delicate estimate.

\begin{lemma}\label{411} Let Assumption \ref{22} be valid, and  $z_\tau\in \h_\tau^{1}$ with $\|z_\tau\|_{\h_\tau^{1/3}}\leq R$. Then
\begin{equation*}
\|U(t, \tau)z_\tau\|^2_{\h_t^1}\leq \mathcal Q\left(R+\|z_\tau\|_{\h_\tau^1}\right)e^{-\omega (t-\tau)}+ \mathcal Q(R),\ \ \forall t\geq \tau.
\end{equation*}
\end{lemma}

\begin{proof}
For any $t\geq \tau$, we define the functionals of the solutions  $(u(t), \partial_t u(t), \eta^t)= U(t, \tau)z_\tau$ as follows:
\begin{align}
&\E_1(t, \tau)= \frac12\|U(t, \tau)z_\tau\|^2_{\h^1_t},\nonumber\\
&\mathcal L_1(t) = L_1(t)+ \|\eta^t\|^2_{{\mathcal M}^1_t}, \label{4.35}\\
& \Lambda_1(t)= \mathcal L_1(t)+ 2\e \Big[\Phi(t)+ 4\Psi(t)\Big],\label{4.36}
\end{align}
where $\e\in (0, 1]$,  $\Phi$ and $\Psi$ are defined as in \eqref{4.6}-\eqref{4.7}, with $
\left(p(t), \partial_t p(t), \psi^t\right)=  \left(A^{1/2}u(t), \partial_t A^{1/2}u(t), A^{1/2}\eta^t\right)$, and the functional
\begin{equation*}
L_1(t)=\|u(t)\|^2_2+ \|\partial_tu(t)\|^2_1+ 2\langle \gamma(t),  A^{1/2}u\rangle, \ \ \hbox{with}\ \   \gamma(t)=A^{1/2} \left(f(u)-g\right).
\end{equation*}
By \eqref{4.20},
\begin{equation*}
2|\langle \gamma(t),  A^{1/2}u\rangle|\leq 2\|\gamma(t)\|_{-1} \|A^{1/2}u\|_1\leq C\left[\|f(u)\|+\|g\|\right]\|u\|_2\leq \frac14 \|u(t)\|^2_2+\mathcal Q(R).
\end{equation*}
Consequently,
\begin{equation}\label{4.37}
\frac32 \E_1(t, \tau)- \mathcal Q(R)\leq \mathcal L_1(t)\leq \frac52\E_1(t, \tau)+ \mathcal Q(R).
\end{equation}
It follows from \eqref{4.8} that
 \begin{equation}\label{4.33'}
\left|\Phi(t)\right|+ \left|\Psi(t)\right|\leq C\E_1(t, \tau), \ \ \forall t\geq \tau,
\end{equation}
which combining with    \eqref{4.36}-\eqref{4.37} gives
\begin{equation}\label{4.39}
 \E_1(t, \tau)- \mathcal Q(R)\leq  \Lambda_1(t)\leq 3 \E_1(t, \tau)+ \mathcal Q(R)
\end{equation}
for $\e\in (0, 1]$ suitably small. Taking the multiplier $2A \partial_t u $ in Eq. \eqref{1.1} gives
\begin{equation*}
\begin{split}
 \frac{\d}{\d t}L_1(t)+ 2\langle \eta^t, \partial_tu(t)\rangle _{\mathcal M_t^1}
=\ & 2\left \langle f'(u(t))\partial_t u(t), A u(t)\right \rangle\\
\leq\ & C\left(1+ \|u(t)\|^2_{L^{18}}\right)\|\partial_t u(t)\|_{L^{18/7}}\|Au(t)\|\\
\leq\ & \mathcal Q(R)\|u(t)\|_2\leq \alpha \|u(t)\|^2_2 +\frac{\mathcal Q(R)}{\alpha}, \ \ \forall \alpha\in (0, 1],\ t\geq \tau,
\end{split}
\end{equation*}
where we have used condition \eqref{2.1},  Sobolev embedding   \eqref{4.24} and formula \eqref{4.32}.
  Integrating above inequality  over $[a, b]$, with $b\geq a\geq \tau$, gives
\begin{equation*}
 L_1(b)+ 2\int_a^b\langle \eta^t, \partial_t u(t)\rangle _{\mathcal M_t^1}\d t
\leq L_1(a)+ \alpha\int_a^b  \|u(t)\|^2_2\d t +  \frac{\mathcal Q(R)}{\alpha}(b-a), \ \ \forall \alpha\in (0, 1].
\end{equation*}
Applying  Lemma \ref{42} (taking $\sigma=1$ there) yields
\begin{equation*}
 \|\eta^b\|^2_{\mathcal M_b^1}-\int_a^b\int_0^\infty\left[\partial_t\mu_t(s)+ \partial_s \mu_t(s) \right]\|\eta^t(s)\|^2_2\d s\d t
\leq\|\eta^a\|^2_{\mathcal M_a^1}+2\int_a^b\langle\partial_t u(t), \eta^t\rangle_{\mathcal M_t^1}\d t.
\end{equation*}
Adding above  two inequalities together turns out
\begin{equation}\label{4.38}
\begin{split}
& \mathcal L_1(b)- \int_a^b\int_0^\infty\left[\partial_t\mu_t(s)+ \partial_s \mu_t(s) \right]\|\eta^t(s)\|^2_2\d s\d t\\
\leq\ &\mathcal L_1(a)+ \alpha\int_a^b \|u(t)\|^2_2\d t +  \frac{\mathcal Q(R)}{\alpha}(b-a), \ \ \forall \alpha\in (0, 1].
\end{split}
\end{equation}

Now, we give the estimates of the last term in the right hand side of the corresponding formulas \eqref{4.4}-\eqref{4.5}, respectively. By \eqref{4.20},
\begin{equation*}
-2\int_a^b \langle \gamma(t), A^{1/2}u(t)\rangle \d t\leq\frac {1}{20} \int_a^b \| u(t)\|^2_2\d t+ \mathcal Q(R)(b-a),
\end{equation*}
and by conditions $(M_1)$, $(M_5)$,
\begin{align*}
&4\int_a^b \frac{2}{\kappa(t)}\int_0^\infty\mu_t(s)\langle A^{1/2}\eta^t(s), \gamma(t)\rangle \d s\d t\\
\leq \  & 8\int_a^b \frac{1}{\kappa(t)}\|\gamma(t)\|_{-1} \left(\int_0^\infty\mu_t(s)\|\eta^t(s)\|_2 \d s\right)\d t\\
\leq \ &  \mathcal Q(R)(b-a)+ \mathcal Q(R) \int_a^b \kappa(t)\|\eta^t\|^2_{\mathcal M_t^{1/3}}\d t.
\end{align*}
Exploiting  \eqref{4.4}-\eqref{4.5} (taking  $\varpi=\frac{1}{20}$  there)  and making use of above two estimates, we obtain
\begin{equation}\label{4.38'}
\begin{split}
& \Phi(b)+ 4\Psi(b)+\int_a^b\Lambda_1(t)\d t +\frac15\int_a^b \| u(t)\|^2_2\d t \\
\leq\ & \Phi(a)+ 4\Psi(a)- 4M\int_a^b\int_0^\infty\left[\partial_t\mu_t(s)+ \partial_s \mu_t(s) \right]\|\eta^t(s)\|^2_2\d s\d t\\
&+\   \mathcal Q(R) \int_a^b \kappa(t)\|\eta^t\|^2_{\mathcal M_t^1}\d t + \mathcal Q(R)(b-a).
\end{split}
\end{equation}
   Taking  $\alpha=\frac{2\e}{5}$, the combination of  \eqref{4.38} and \eqref{4.38'} gives
\begin{equation}\label{4.39'}
  \Lambda_1(b)+2\e \int_a^b\Lambda_1(t)\d t+\mathcal J_1
\leq  \Lambda_1(a)+ \frac{\mathcal Q(R)}{\e}(b-a),
\end{equation}
where we have used condition $(M_4)$ and the fact that
\begin{equation*}
\begin{split}
\mathcal J_1=& -(1-8\e M)\int_a^b\int_0^\infty\left[\partial_t\mu_t(s)+ \partial_s \mu_t(s) \right]\|\eta^t(s)\|^2_2\d s\d t-2\e \mathcal Q(R) \int_a^b \kappa(t)\|\eta^t\|^2_{\mathcal M_t^1}\d t\\
\geq &\left(\delta(1-8\e M)-2\e \mathcal Q(R)\right)\int_a^b \kappa(t)\|\eta^t\|^2_{\mathcal M_t^1}\d t\geq 0
\end{split}
\end{equation*}
for $\e>0$ sufficiently  small.
Applying Lemma \ref{41} (with $q_1=0, q_2=\frac{\mathcal Q(R)}{\e}$ there)  to \eqref{4.39'} and making use of \eqref{4.39} give the conclusion of   Lemma \ref{411}.
\end{proof}
\bigskip

\begin{lemma}\label{412} Let Assumption \ref{22} be valid. Then for any $z_{1\tau}, z_{2 \tau}\in \h_\tau^1$ with $\|z_{i\tau}\|_{\h_{\tau}^1}\leq R$, $i=1, 2$,
\begin{equation*}
\|U(t, \tau)z_{1\tau}- U(t, \tau)z_{2\tau}\|^2_{\h_t}\leq  C e^{-\kappa (t-\tau)} \|z_{1\tau}-z_{2\tau}\|^2_{\h_\tau}+ \mathcal Q(R) e^{t-\tau} \int_\tau^t \|\bar u(s)\|^2\d s,
\end{equation*}
where $C$ and $\kappa$ are  positive constants independent of $R$, and
\begin{equation*}
\bar z(t)=\left(\bar u(t), \partial_t \bar u(t), \bar \eta^t\right)=z_1(t)-z_2(t)\ \ \hbox{and}\ \ z_i(t)=\left(u_i(t), \partial_t u_i(t), \eta_i^t\right)= U(t, \tau)z_{i \tau}, \ i=1,2.
\end{equation*}
\end{lemma}

 \begin{remark} Lemma \ref{412} implies that the process $U(t, \tau): \mathbb{B}^1_\tau(R)\subset\h_\tau \rightarrow \h_t$ is  quasi-stable for all $R>0$.
\end{remark}

\begin{proof}[\textbf{Proof of Lemma \ref{412}}] We still split the solution $U(t, \tau)z_{i \tau}$  into the sum
\begin{equation*}
\begin{split}
U(t, \tau)z_{i \tau} &= U_0(t, \tau)z_{i \tau}+ U_1(t, \tau)z_{i \tau}\\
& =\left(v_i(t), \partial_t v_i(t), \xi_i^t \right)+\left(w_i(t), \partial_t w_i(t), \zeta_i^t \right), \ \ i=1, 2,
\end{split}
\end{equation*}
where $U_0(t, \tau)z_{i \tau}$ and $U_1(t, \tau)z_{i \tau}$ solves problem  \eqref{4.13} and \eqref{4.14} (with $f_0=0$ and  $f_1=f$ there),  respectively.
Then  $\bar v =v_1 -v_2 $ solves
\begin{equation}\label{4.40}
\left\{
  \begin{array}{ll}
    \partial_{tt}\bar v+A\bar v +\int_0^\infty\mu_t(s)A\bar\xi^t(s)\d s=0,\ \ t>\tau,  \\
    \left(\bar v(\tau), \partial_t \bar v(\tau), \bar \xi^\tau \right)=z_{1\tau}-z_{2\tau},
  \end{array}
\right.
\end{equation}
with
\begin{equation*}
\bar \xi^t(s)=\left\{
            \begin{array}{ll}
              \bar v(t)-\bar v(t-s), & s\leq t-\tau, \\
              \bar \xi_\tau(s-t+\tau)+\bar v(t)-\bar v_\tau, & s>t-\tau.
            \end{array}
          \right.
\end{equation*}
And   $\bar w=w_1 -w_2 $ solves
\begin{equation}\label{4.45}
\left\{
  \begin{array}{ll}
    \partial_{tt}\bar w +A\bar w +\int_0^\infty\mu_t(s)A\bar\zeta^t(s)\d s+f(u_1)-f (u_2) =0,\ \ t>\tau, \\
    \left(\bar w(\tau), \partial_t \bar w(\tau), \bar \zeta^\tau \right)=0,
  \end{array}
\right.
\end{equation}
with
\begin{equation*}
\bar \zeta^t(s)=\left\{
            \begin{array}{ll}
              \bar w(t)-\bar w(t-s), & s\leq t-\tau, \\
              \bar \zeta_\tau(s-t+\tau)+\bar w(t)-\bar w_\tau, & s>t-\tau.
            \end{array}
          \right.
\end{equation*}
It follows from Lemmas \ref{410} and \ref{411} that
\begin{equation}\label{4.41}
\|z_i(t)\|_{\h^1_t}+ \|U_0(t, \tau)z_{i \tau}\|_{\h^1_t}+ \|U_1(t, \tau)z_{i \tau}\|_{\h^1_t}\leq \mathcal Q(R), \ \ t\geq \tau.
\end{equation}
 Hence  we can take the multiplier $2\partial_t \bar v$ in Eq. \eqref{4.40} and obtain
\begin{align}
&\|\bar v(b)\|^2_1+ \|\partial_t \bar v(b)\|^2+2\int_a^b\langle \bar \xi^t, \partial_t \bar v(t)\rangle_{\mathcal M_t}\d t=\|\bar v(a)\|^2_1+ \|\partial_t \bar v(a)\|^2,\ \forall b\geq a\geq \tau.\label{4.41'}
\end{align}
By Lemma \ref{42} (taking $\sigma=0$ there),
\begin{equation}\label{4.43'}
\|\bar\xi^b\|^2_{\mathcal M_b}-\int_a^b\int_0^\infty\left[\partial_t\mu_t(s)+ \partial_s \mu_t(s) \right]\|\bar\xi^t(s)\|^2_1\d s\d t
\leq\|\bar\xi^a\|^2_{\mathcal M_a}+2\int_a^b\langle\partial_t \bar v(t), \bar\xi^t\rangle_{\mathcal M_t}\d t.
\end{equation}
The combination of   \eqref{4.41'} and \eqref{4.43'} gives
\begin{equation}\label{4.42}
\begin{split}
&\|(\bar v(b), \partial_t \bar v(b), \bar \xi^b )\|^2_{\h_b}-\int_a^b\int_0^\infty\left[\partial_t\mu_t(s)+ \partial_s \mu_t(s) \right]\|\bar\xi^t(s)\|^2_1\d s\d t\\
\leq\ & \|\left(\bar v(a), \partial_t \bar v(a), \bar \xi^a \right)\|^2_{\h_a},\ \forall b\geq a\geq \tau.
\end{split}
\end{equation}

In order to obtain the sufficient regularity of the solutions needed for applying Lemma \ref{43}, we use the following approximating technique.
We denote by
\begin{equation*}
\left(v_{i n}(t), \partial_tv_{i n}(t), \xi^t_{in} \right), \  \  i=1, 2
\end{equation*}
the   Galerkin approximations of  $(v_{i}(t), \partial_tv_{i}(t), \xi^t_{i}),   i=1, 2$, with initial data
\begin{equation}\label{06091}
\left(v_{i n}(\tau), \partial_tv_{i n}(\tau), \xi^\tau_{in} \right)\rightarrow z_{i\tau}\ \ \hbox{in}\ \ \h_\tau, \ \ i=1, 2.
\end{equation}
It follows from \eqref{4.42}-\eqref{06091} and condition $(M_4)$ that
\begin{equation}\label{06092}
\begin{split}
& \lim_{n\rightarrow \infty} \|\left(v_{i n}(t), \partial_tv_{i n}(t), \xi^t_{in} \right)-\left(v_i(t), \partial_t v_i(t), \xi_i^t \right)\|_{\h_t}\\
\leq\ & \lim_{n\rightarrow \infty} \|\left(v_{i n}(\tau), \partial_tv_{i n}(\tau), \xi^\tau_{in} \right)- z_{i\tau}\|_{\h_\tau}=0,\ \ \forall t\geq \tau, \ i=1, 2.
\end{split}
\end{equation}

For every $n\in\mathbb N$, let $\bar v_{n}=v_{1n}-v_{2n}$, $\bar \xi_{\tau n}=\xi^\tau_{1n}-\xi^\tau_{in}$ and
\begin{equation*}
\bar \xi_n^t(s)=\left\{
            \begin{array}{ll}
              \bar v_n(t)-\bar v_n(t-s), & s\leq t-\tau, \\
              \bar \xi_{\tau n}(s-t+\tau)+\bar v_n(t)-\bar v_{\tau n}, & s>t-\tau.
            \end{array}
          \right.
\end{equation*}
Obviously, formula \eqref{4.42} holds for all $\left( \bar v_{n}, \partial_t\bar v_{n}, \bar \xi^t_n\right)$, $n\in\mathbb N$.

For every $\e\in (0, 1]$, we introduce the functional
\begin{equation*}
\Lambda^n_v(t)=\|(\bar v_n(t), \partial_t \bar v_n(t), \bar \xi_n^t )\|^2_{\h_t}+ 2\e\Big[\Phi_n(t)+ 4\Psi_n(t)\Big], \ \ n\in\mathbb N,
\end{equation*}
where the functionals $\Phi_n$ and $\Psi_n$ are defined by formulas   \eqref{4.6}-\eqref{4.7}, with
\begin{equation*}
\left(p(t), \partial_t p(t), \psi^t\right)= \left(\bar v_n(t), \partial_t \bar v_n(t), \bar \xi_n^t \right)\ \ \hbox{and}\ \ \gamma(t)=0\ \ \hbox{in \eqref{4.1}}.
\end{equation*}
Thus, it follows from formula  \eqref{4.8} that
\begin{equation}\label{4.43}
\frac12 \|(\bar v_n(t), \partial_t \bar v_n(t), \bar \xi_n^t )\|^2_{\h_t}\leq \Lambda^n_v(t)\leq \frac32 \|(\bar v_n(t), \partial_t \bar v_n(t), \bar \xi_n^t )\|^2_{\h_t}, \ \ n\in\mathbb N
\end{equation}
for $\e>0$ sufficiently small. And by  Lemma \ref{43} (taking $\varpi=1/20$  and $\gamma(t)=0$ there) and a simple calculation we obtain
\begin{equation}\label{4.47'}
\begin{split}
& \Phi_n(b)+ 4\Psi_n(b)+\frac74\int_a^b\|\bar v_n(t)\|^2_1\d t +2\int_a^b \| \partial_t \bar v_n(t)\|^2\d t \\
\leq\ & \Phi_n(a)+ 4\Psi_n(a)- 4M\int_a^b\int_0^\infty\left[\partial_t\mu_t(s)+ \partial_s \mu_t(s) \right]\|\bar \xi_n^t(s)\|^2_1\d s\d t+ C\int_a^b \kappa(t)\|\bar \xi_n^t\|^2_{\mathcal M_t}\d t.
\end{split}
\end{equation}
The combination of \eqref{4.42} and \eqref{4.47'} gives
\begin{equation*}
\Lambda^n_v(b)+2\e\int_a^b \Lambda^n_v(t)\d t \leq \Lambda^n_v(a)
\end{equation*}
for  $\e>0$ suitably small, where we have used condition $(M_4)$. Hence  applying Lemma \ref{41}, with $q_1=q_2=0$ there, and making use of \eqref{4.43}, we obtain
\begin{equation}\label{4.44'}
\|(\bar v_n(t), \partial_t \bar v_n(t), \bar \xi_n^t )\|^2_{\h_t}\leq 3 e^{-\e(t-\tau)}\|(\bar v_n(\tau), \partial_t \bar v_n(\tau), \bar \xi_n^\tau )\|^2_{\h_\tau}, \ \ \forall t\geq \tau, \ n\in\mathbb N.
\end{equation}
Thus, by \eqref{06091}-\eqref{06092} and formula \eqref{4.44'}, we have
\begin{equation}\label{4.44}
\begin{split}
\|(\bar v(t), \partial_t \bar v(t), \bar \xi^t )\|^2_{\h_t}
=\ & \lim_{n\rightarrow \infty} \|(\bar v_n(t), \partial_t \bar v_n(t), \bar \xi_n^t )\|^2_{\h_t}\\
\leq\ & 3 e^{-\e(t-\tau)} \lim_{n\rightarrow \infty}\|(\bar v_n(\tau), \partial_t \bar v_n(\tau), \bar \xi_n^\tau )\|^2_{\h_\tau}\\
=\ & 3 e^{-\e(t-\tau)}\|z_{1\tau}-z_{2\tau}\|^2_{\h_\tau}.
\end{split}
\end{equation}
Taking into account   estimate \eqref{4.41}, we can use multiplier $2\partial_t \bar w$ in Eq. \eqref{4.45} and arrive at
\begin{align*}
& \frac{\d}{\d t}\left[\|\bar w(t)\|^2_1+ \|\partial_t \bar w(t)\|^2\right]+2\langle \bar \zeta^t, \partial_t \bar w(t)\rangle_{\mathcal M_t}\\
=\ & 2\langle f\left(u_2\right)-f\left(u_1\right), \partial_t \bar w(t)\rangle\\
\leq\ & C\left(1+ \|u_1(t)\|^2_{L^\infty}+ \|u_2(t)\|^2_{L^\infty}\right)\|\bar u(t)\| \|\partial_t \bar w(t)\|\\
\leq\ & \mathcal Q(R) \|\bar u(t)\|^2+ \|\partial_t \bar w(t)\|^2,\\
\end{align*}
which implies that
\begin{equation*}
 \|\bar w(t)\|^2_1+ \|\partial_t \bar w(t)\|^2+2\int_\tau^t \langle \bar \zeta^s, \partial_s \bar w(s)\rangle_{\mathcal M_s}ds
\leq\mathcal Q(R)\int_\tau^t (\|\bar u(s)\|^2+ \|\partial_t \bar w(s)\|^2)ds,\ \ \forall t\geq \tau,
\end{equation*}
where we have used condition \eqref{2.1}, estimate \eqref{4.41} and the Sobolev embedding $H^2 \hookrightarrow L^\infty$. Thus,  making use of condition $(M_4)$ and Lemma \ref{42} (with $\sigma=0$ there), we obtain
 \begin{equation*}
  \|(\bar w(t), \partial_t \bar w(t), \bar \zeta^t )\|^2_{\h_t}\leq \mathcal Q(R)\int_\tau^t \|\bar u(s)\|^2\d s+\int_\tau^t\|(\bar w(s), \partial_t \bar w(s), \bar \zeta^s )\|^2_{\h_s}\d s.
 \end{equation*}
 Applying the Gronwall inequality to above estimate gives
\begin{equation}\label{4.46}
\|(\bar w(t), \partial_t \bar w(t), \bar \zeta^t )\|^2_{\h_t}\leq \mathcal Q(R)e^{t-\tau}\int_\tau^t \|\bar u(s)\|^2\d s, \ \ \forall t\geq \tau.
\end{equation}
The combination of \eqref{4.44} and \eqref{4.46} yields
\begin{equation*}
\begin{split}
&\|U(t, \tau)z_{1\tau}-U(t, \tau)z_{2\tau}\|^2_{\h_t}\\
\leq \ & C\left[\|(\bar v(t), \partial_t \bar v(t), \bar \xi^t )\|^2_{\h_t}+ \|(\bar w(t), \partial_t \bar w(t), \bar \zeta^t )\|^2_{\h_t} \right]\\
\leq\ & Ce^{-\e(t-\tau)}\|z_{1\tau}-z_{2\tau}\|^2_{\h_\tau}+ \mathcal Q(R)e^{t-\tau}\int_\tau^t \|\bar u(s)\|^2\d s, \ \ t\geq \tau,
\end{split}
\end{equation*}
which completes the proof.
\end{proof}

\section{Proof of the main result}
The purpose of this section is to  prove Theorem \ref{eaexist} by applying the abstract criteria obtained in Section 2. This argument  is  challenging because of the hyperbolicity of the  problem which results in without any additional regularity of its solutions, so we put forward a new technique   to overcome this difficulty. To this end, we first  establish a specially pullback attracting family.

\begin{lemma}\label{52}Let Assumption \ref{22} be valid. Then there exists a family $\{B(t)\}_{t\in\r}$, with $B(t)\subset \h_t$ for each $t\in \r$,  possessing the following properties:
\begin{description}
  \item (i)  for every $t\in\r$, the section $B(t)$ is closed in $\h_t$ and
        \begin{equation}\label{5.1}
        B(t)\subset \mathbb B_t(\mathcal R_0)\cap \mathbb B^1_t(\mathcal R)
         \end{equation}
       for some constants $\mathcal R>0$ and $\mathcal R_0>R_1$, where $R_1$ is given by Lemma \ref{45};
  \item (ii)  there exist positive constants $\kappa$ and $\tau_1$ such that
   \begin{equation}\label{5.2}
   \mathrm{dist}_{\h_t}\left(U(t, \tau)\mathbb B_\tau(R_1), B(t)\right)\leq \mathcal Q(R_1)e^{-\kappa(t-\tau)}, \ \ \forall \tau\leq t-\tau_1, \ t\in\r;
   \end{equation}
   \item (iii)  there exists a positive constant $T_1$ such that
    \begin{equation}\label{5.3}
      U(t, \tau)B(\tau)\subset B(t), \ \ \forall \tau\leq t-T_1,\ t\in\r.
      \end{equation}
\end{description}
\end{lemma}
\begin{proof} For any $\tau\in\r$ and $z_\tau\in \mathbb B_\tau(R_1)$,  it follows from Lemma \ref{48} that
\begin{align*}
& \|U_0(t, \tau)z_\tau\|^2_{\h_t}\leq \mathcal Q(R_1)e^{-\omega (t-\tau)}\ \ \hbox{and}\ \  \|U_1(t, \tau)z_\tau\|^2_{\h^{1/3}_t}\leq \mathcal Q(R_1),\ \ \forall t\geq \tau,
\end{align*}
  which implies that there exists a positive constant $\mathcal R_1$ depending only on $R_1$  such that
 \begin{equation}\label{5.4}
   \mathrm{dist}_{\h_t}\left(U(t, \tau)\mathbb B_\tau(R_1),\  \mathbb B^{1/3}_t(\mathcal R_1)\right)\leq \mathcal Q(R_1)e^{-\omega(t-\tau)}, \ \ \forall t\geq \tau.
   \end{equation}

Similarly, we infer from Lemma \ref{410} that, for any $z_\tau\in \mathbb B^{1/3}_\tau(\mathcal R_1)$,
\begin{align*}
&\|U_0(t, \tau)z_\tau\|^2_{\h_t}\leq \mathcal Q(\mathcal R_1) e^{-\omega (t-\tau)}\ \ \hbox{and}\ \  \|U_1(t, \tau)z_\tau\|^2_{\h^1_t}\leq \mathcal Q(\mathcal R_1),\ \ t\geq \tau.
\end{align*}
 Since $\mathcal R_1$ depends only on $R_1$, we can find a constant $\mathcal R_2$ depending only on $R_1$ such that
\begin{equation}\label{5.5}
   \mathrm{dist}_{\h_t}\left(U(t, \tau)\mathbb B^{1/3}_\tau(\mathcal R_1),\  \mathbb B^1_t(\mathcal R_2)\right)\leq \mathcal Q(R_1)e^{-\omega(t-\tau)}, \ \ \forall t\geq \tau.
   \end{equation}

It follows from Definition \ref{26} and Lemma \ref{45} that there exists a positive constant $e(R_1)$  depending only on $R_1$ such that
\begin{equation}\label{5.6}
U(t, \tau)\mathbb B_\tau (R_1)\subset \mathbb B_t(R_1), \ \ \forall \tau\leq t-e(R_1).
\end{equation}

Let $\theta=\frac{\omega}{\mathcal Q(R_1)+2\omega}$. Obviously,
\begin{equation*}
\theta\in (0, 1)\ \ \hbox{and}\ \ -\omega \theta=-\omega+ \left(\mathcal Q(R_1)+\omega\right)\theta.
\end{equation*}
We infer from formula \eqref{5.6} that
\begin{equation}\label{5.7}
U\left((1-\theta)t+\theta\tau, \tau\right)\mathbb B_\tau (R_1)\subset \mathbb B_{(1-\theta)t+\theta\tau}(R_1), \ \ \forall \tau\leq t-e_1,
\end{equation}
where $e_1=\frac{e(R_1)}{1-\theta}>0$. Thus, it follows from Theorem \ref{existence} and formula \eqref{5.4}-\eqref{5.7} that
\begin{align}
&\mathrm{dist}_{\h_t}\left(U(t, \tau)\mathbb B_\tau (R_1), \mathbb B_t^1(\mathcal R_2)\right)\nonumber\\
\leq \ &\mathrm{dist}_{\h_t}\left(U\left(t, (1-\theta)t+\theta \tau\right)U\left((1-\theta)t+\theta \tau, \tau\right)\mathbb B_\tau (R_1),  U\left(t, (1-\theta)t+\theta \tau\right)\mathbb B_{(1-\theta)t+\theta \tau}^{1/3}(\mathcal R_1)\right)\nonumber\\
&+ \mathrm{dist}_{\h_t}\left(U\left(t, (1-\theta)t+\theta \tau\right)\mathbb B_{(1-\theta)t+\theta \tau}^{1/3}(\mathcal R_1),  \mathbb B_t^1(\mathcal R_2) \right)\label{5.8}\\
\leq\ &\mathcal Q(R_1) \exp\{\mathcal Q(R_1)\theta (t-\tau)\}\mathrm{dist}_{\h_{(1-\theta)t+\theta \tau}}\left(U\left((1-\theta)t+\theta \tau, \tau\right)\mathbb B_\tau (R_1),  \mathbb B_{(1-\theta)t+\theta \tau}^{1/3}(\mathcal R_1)\right)\nonumber\\
&+ \mathcal Q(R_1)e^{-\omega \theta (t-\tau)}\nonumber\\
\leq \ & \mathcal Q(R_1)\exp\{\left[-\omega+ \left(\mathcal Q(R_1)+ \omega\right)\theta \right](t-\tau)\}+ \mathcal Q(R_1)e^{-\omega \theta (t-\tau)}\nonumber\\
\leq \ & \mathcal Q(R_1)e^{-\omega \theta (t-\tau)},\ \ \forall \tau\leq t-e_1.\nonumber
\end{align}

For every $z\in \mathbb B_t^1(\mathcal R_2)$,
\begin{equation*}
\|z\|_{\h_t}\leq \lambda_1^{-1/2} \|z\|_{\h_t^1}\leq \lambda_1^{-1/2} \mathcal R_2,\ \ \forall t\in\r,
\end{equation*}
which implies
\begin{equation}\label{5.9}
\mathbb B_t^1\left(\mathcal R_2\right)\subset \mathbb B_t\left(\lambda_1^{-1/2}\mathcal R_2\right) \subset \mathbb B_t(\mathcal R_3) \ \ \hbox{and}\ \ \mathbb B_t(R_1) \subset \mathbb B_t(\mathcal R_3),\ \ \forall  t\in\r,
\end{equation}
 where $\mathcal R_3=R_1+ \lambda_1^{-1/2}\mathcal R_2$ depends only on $R_1$. By Lemma \ref{45} and formula \eqref{5.9}, there exists a constant $e_2>0$ such that
\begin{equation}\label{5.10}
  U(t, \tau)\mathbb B_\tau(\mathcal R_3)\subset\mathbb B_t(R_1) \subset \mathbb B_t(\mathcal R_3), \ \ \forall \tau\leq t-e_2.
\end{equation}
It follows from Lemma \ref{49} that for any  $z_\tau\in \mathbb B_\tau(\mathcal R_3)\cap \h_\tau^{1/3}$,
\begin{equation}\label{5.11}
\|U(t, \tau)z_\tau\|^2_{\h_t^{1/3}}\leq \mathcal Q\left(\mathcal R_3+\|z_\tau\|_{\h_\tau^{1/3}}\right)e^{-\omega(t-\tau)}+\mathcal R_4, \ \ \forall t\geq \tau,
\end{equation}
where the positive constant $\mathcal R_4$ depends only on $R_1$.

Similarly, for every  $z\in \mathbb B_t^1(\mathcal R_2)$, we have
\begin{equation*}
\|z\|_{\h^{1/3}_t}\leq \lambda_1^{-1/3} \|z\|_{\h_t^1}\leq \lambda_1^{-1/3}\mathcal R_2, \ \ \forall t\in\r,
\end{equation*}
which implies
\begin{equation}\label{5.12}
\mathbb B_t^1\left(\mathcal R_2\right)\subset \mathbb B^{1/3}_t\left(\lambda_1^{-1/3} \mathcal R_2\right) \subset \mathbb B^{1/3}_t(\mathcal R_5),\ \ \forall t\in\r,
\end{equation}
where $\mathcal R_5=\mathcal R_4+ \lambda_1^{-1/3} \mathcal R_2$ depends only on $R_1$. It follows from  formula \eqref{5.11} that there exists a positive constant $e_3$ such that
\begin{equation}\label{5.13}
U(t, \tau)\left[\mathbb B_\tau\left(\mathcal R_3\right) \cap \mathbb B^{1/3}_\tau(\mathcal R_5) \right]\subset \mathbb B^{1/3}_t(\mathcal R_5), \ \ \forall \tau\leq t-e_3.
\end{equation}

Lemma  \ref{411} shows that for any $z_\tau\in
\mathbb B^{1/3}_\tau(\mathcal R_5)\cap \h_\tau^1$,
\begin{equation}\label{5.14}
\|U(t, \tau)z_\tau\|^2_{\h_t^1}\leq \mathcal Q\left(\mathcal R_5+\|z_\tau\|_{\h_\tau^1}\right)e^{-\omega(t-\tau)}+\mathcal R_6, \ \ \forall t\geq \tau,
\end{equation}
where the  positive constant $\mathcal R_6$ depends only on $R_1$. Obviously,
\begin{equation}\label{5.15}
\mathbb B^1_t(\mathcal R_2)\subset \mathbb B^1_t(\mathcal R_7)\ \ \hbox{with}\ \ \mathcal R_7=\mathcal R_2+ \mathcal R_6, \ \ \forall t\in\r.
\end{equation}
 Thus formula \eqref{5.14} implies that there is a positive constant $e_4$ such that
\begin{equation}\label{5.16}
U(t, \tau)\left[\mathbb B^{1/3}_\tau\left(\mathcal R_5\right) \cap \mathbb B^1_\tau(\mathcal R_7) \right]\subset \mathbb B^1_t(\mathcal R_7), \ \ \forall \tau\leq t-e_4.
\end{equation}

Let
\begin{equation*}
B(t)=\mathbb B_t\left(\mathcal R_3\right) \cap \mathbb B^{1/3}_t(\mathcal R_5) \cap \mathbb B^1_t(\mathcal R_7), \ \ \forall t\in\r.
\end{equation*}
We show that $\{B(t)\}_{t\in\r}$ is the desired family.\medskip

(i)\ Obviously, for every $t\in\r$, $B(t)$ is closed in $\h_t$ and
  \begin{equation*}
B(t)\subset \mathbb B_t\left(\mathcal R_3\right)  \cap \mathbb B^1_t(\mathcal R_7),
\end{equation*}
that is, conclusion \eqref{5.1} is valid, with $\mathcal R_0= \mathcal R_3> R_1$ and $\mathcal R=\mathcal R_7$.

(ii)\ It follows from formulas \eqref{5.9}, \eqref{5.12} and \eqref{5.15} that
$\mathbb B^1_t(\mathcal R_2)\subset B(t)$ holds for all $t\in\r$. Then  we infer from estimates \eqref{5.8} that
\begin{equation*}
\begin{split}
\mathrm{dist}_{\h_t}\left(U(t, \tau)\mathbb B_\tau (R_1), B(t)\right)
&\leq \mathrm{dist}_{\h_t}\left(U(t, \tau)\mathbb B_\tau (R_1), \mathbb B^1_t(\mathcal R_2)\right)\\
&\leq \mathcal Q(R_1) e^{-\omega\theta (t-\tau)}, \ \ \forall \tau\leq t-e_1,
\end{split}
\end{equation*}
that is, formula \eqref{5.2} holds, with   $\kappa=\omega \theta$ and $\tau_1=e_1$.

(iii)\ Taking $T_1=\max\{e_2, e_3, e_4\}$ and making use of formulas \eqref{5.10}, \eqref{5.13} and \eqref{5.16} yield
\begin{align*}
&U(t, \tau)B(\tau)\subset U(t, \tau)\mathbb B_\tau (\mathcal R_3)\subset \mathbb B_t (\mathcal R_3),\\
& U(t, \tau)B(\tau)\subset U(t, \tau)\left[\mathbb B_\tau (\mathcal R_3)\cap
\mathbb B^{1/3}_\tau \left(\mathcal R_5\right)\right]\subset \mathbb B^{1/3}_t\left(\mathcal R_5\right),\\
& U(t, \tau)B(\tau)\subset U(t, \tau)\left[\mathbb B^{1/3}_\tau (\mathcal R_5)\cap \mathbb B^1_\tau \left(\mathcal R_7\right)\right]\subset \mathbb B^1_t\left(\mathcal R_7\right), \ \ \forall \tau\leq t-T_1,\  t\in\r.
\end{align*}
  Therefore,
\begin{equation*}
U(t, \tau)B(\tau)\subset \mathbb B_t\left(\mathcal R_3\right) \cap \mathbb B^{1/3}_t(\mathcal R_5) \cap \mathbb B^1_t(\mathcal R_7)=B(t), \ \ \forall \tau\leq t-T_1, \ t\in\r.
\end{equation*}
This completes the proof.
\end{proof}
\bigskip

\begin{proof}[\textbf{Proof of Theorem \ref{eaexist}}] It follows from Lemma \ref{52} that the family $\{B(t)\}_{t\in\r}$ is uniformly bounded in $\{\h_t\}_{t\in\r}$, $B(t)$ is closed in $\h_t$ for each $t\in\r$, and there exists a positive constant $T>T_1$ such that $\eta^2= C e^{-\kappa T}<\frac14$ and
\begin{equation*}
U(t, t-\tau)B(t-\tau)\subset B(t), \ \ \forall t\in\r, \ \tau\geq T,
\end{equation*}
where $T_1$ is as shown in  Lemma \ref{52}. It follows from formula \eqref{2.6} that for any $t\in\r$,
\begin{equation}\label{5.17}
\|U(t, t-\tau)z_1- U(t, t-\tau)z_1\|_{\h_t}\leq L_1 \|z_1-z_2\|_{\h_{t-\tau}}, \ \ \forall z_1, z_2\in \mathbb B_t(\mathcal R_0),\ \tau\in [0, T],
\end{equation}
where the positive constant $L_1$ depends only on $\mathcal R_0$ and $T$.

Define the   space
\begin{equation*}
Z=\left\{u\in L^2\left(0, T; H^1\right)\ |\ \partial_t u\in L^2\left(0, T; H\right)\right\}
\end{equation*}
equipped with the  norm
\begin{equation*}
\|u\|_Z=\|\left(u, \partial_t u\right)\|_{L^2\left(0, T; H^1\times H\right)}.
\end{equation*}
Obviously, $Z$ is a Banach space. And the functional
\[ n_Z(u)=\mathcal Q\left(\mathcal R_0+T\right)\|u\|_{L^2\left(0, T; H\right)}\]
 is a compact semi-norm on $Z$ (cf. \cite{Simon}). For any given $t\in\r$, we define the mapping
\begin{equation*}
K_t: B(t-T)\rightarrow Z, \ \ K_t z=u(\cdot+t-T), \ \ \forall z\in B(t-T),
\end{equation*}
where $u(\cdot+t-T)$ means $u(s+t-T), s\in [0, T]$, and
\begin{equation*}
\left(u(s+t-T), \partial_t u(s+t-T), \eta^{s+t-T}\right)=U(s+t-T, t-T)z.
\end{equation*}
 Lemma \ref{412} shows that
\begin{equation*}
\|U(t, t-T)z_1- U(t, t-T)z_2\|_{\h_t}\leq \eta \|z_1-z_2\|_{\h_{t-T}}+n_Z\left(K_t z_1-K_t z_2\right),
\end{equation*}
and we infer from formulas \eqref{2.6} and  \eqref{5.1} that
\begin{equation*}
\begin{split}
\|K_t z_1-K_t z_2\|_Z^2
\leq\ & \int_0^T \|U(s+t-T, t-T)z_1-U(s+t-T, t-T)z_2\|^2_{\h_{s-t+T}} \d s\\
\leq\ & e^{\mathcal Q(\mathcal R_0) T}\|z_1-z_2\|^2_{\h_{t-T}},\ \ \forall z_1, z_2\in B(t-T), \ \ t\in\r.
\end{split}
\end{equation*}
 Thus   the family $\b=\{B(t)\}_{t\in\r}$ satisfies conditions  $(H_1)$-$(H_3)$ of Theorem \ref{ea}.

Moreover, by  Lemma \ref{45}, $\{\mathbb B_t(R_1)\}_{t\in\r}$ is a uniformly  time-dependent  absorbing set of the process  $U(t, \tau)$. And formulas \eqref{5.1}-\eqref{5.2} and \eqref{5.17} show that  $\{\mathbb B_t(R_1)\}_{t\in\r}$ satisfies the conditions of Corollary \ref{ea2}. Therefore,  the process  $U(t, \tau)$ has a time-dependent exponential attractor $\E=\{E(t)\}_{t\in\r}$, with $E(t)\subset B(t)\subset \mathbb B_t^1(\mathcal R)$ for each  $t\in\r$.
\end{proof}

\end{document}